\documentclass[10pt]{amsart}
\usepackage{CJK}
\usepackage{indentfirst}
\usepackage{amscd}
\usepackage{amsmath, amssymb}
\usepackage{amsfonts}
\usepackage{enumerate}

\newcommand{\ve}{\varepsilon}
\newcommand{\ddbar}{\sqrt{-1} \partial \overline{\partial}}

\newcommand{\MA}{Monge-Amp\`{e}re}

\newcommand{\ri}{\rightarrow}
\newcommand{\DP}{Dirichlet problem}
\newcommand{\acm}{almost Hermitian manifold}

\begin{document}
\begin{CJK}{GBK}{song}
\newcounter{theor}
\setcounter{theor}{1}
\newtheorem{claim}{Claim}
\newtheorem{theorem}{Theorem}[section]
\newtheorem{lemma}[theorem]{Lemma}
\newtheorem{corollary}[theorem]{Corollary}
\newtheorem{proposition}[theorem]{Proposition}
\newtheorem{question}{question}[section]
\newtheorem{defn}{Definition}[section]
\newtheorem{remark}{Remark}[section]

\numberwithin{equation}{section}

\title[Monge-Amp\`{e}re type equations on almost Hermitian manifolds]{Monge-Amp\`{e}re type equations on almost Hermitian manifolds }
\author[J. Zhang]{Jiaogen Zhang}
\address{Department of Mathematical Sciences, University of Science and Technology of China, Hefei 230026, People's Republic of China }
\email{zjgmath@mail.ustc.edu.cn}


\begin{abstract}
In this paper we consider the Monge-Amp\`{e}re type equations on compact almost Hermitian manifolds. We derive a priori estimates under the existence of an admissible $\mathcal{C}$-subsolution. Finally, we also obtain an existence theorem if there exists an admissible supersolution.
\end{abstract}
\maketitle

\section{Introduction}
Let $(M,J,\omega)$ be a compact \acm ~of real dimension $2n$. Suppose $\chi$ is a real (1,1) form on $M$.
For each $u\in C^{2}(M,\mathbb{R})$, we will use the shorthand $$\chi_{u}=\chi+\ddbar u.$$ Let
\[
\mathcal{H}(M)=\{u\in C^{2}(M,\mathbb{R}): \chi_{u}>0\}
\]
be the admissible set with respect to $\chi$. In this paper we wish to consider the \MA~type equations for $u$ which can be written in the form
\begin{equation}\label{dp}
\left\{\begin{array}{ll}
u\in \mathcal{H}(M),\\[1mm]
\chi_{u}^{n}=\psi\chi_{u}^{n-m}\wedge \omega^{m}~\textrm{on}~{M}, \\[1mm]
\end{array}\right.
\end{equation}
where $\psi\geq c>0$ is a smooth real valued function on $M$ and $m\in[1, n] $ is a fixed integer.

In contrast to the complex \MA~ equation \cite{CTW}, we need to further assume there exists at least one $\mathcal{C}$-subsolution (see Definition \ref{subsolution}) for (\ref{dp}), which was also specified in \cite{Gab,Gu14}. We have the following theorem.
\begin{theorem}\label{Thm1.1}
Let $(M,J,\omega)$ be a compact \acm ~of real dimension $2n$. Suppose there exists an admissible $\mathcal{C}$-subsolution $\underline{u}$ and $u$ is the smooth solution for  (\ref{dp}). Then we have estimates \[\|u\|_{C^{k,\alpha}}<C,\] where $C$ is a constant depends on $(M,J,\omega)$, $\chi$, $k$, $\alpha$, $\psi$ and $\underline{u}$.
\end{theorem}
We say $v\in C^{2}(M)$ is a supersolution for \eqref{dp} if 
\begin{equation}\label{supersolution}
	\chi_{v}^{n}\leq \psi\chi_{v}^{n-m}\wedge\omega^{m}.
\end{equation}
Analogous to the Hermitian setting in \cite{Sun1}, we can also obtain the following existence theorem under an extra condition.
\begin{theorem}
Let $(M,J,\omega)$ be a compact \acm ~of real dimension $2n$. Suppose there exist an admissible $\mathcal{C}$-subsolution $\underline{u}$ and an admissible supersolution for the equation (\ref{dp}).
Then there exists a pair $(u,b)\in C^{\infty}(M)\times \mathbb{R}$ such that
\begin{equation}\label{dp2}
\left\{\begin{array}{ll}
u\in \mathcal{H}(M), \qquad \sup_{M}u=0,\\[1mm]
\chi_{u}^{n}=e^{b}\psi\chi_{u}^{n-m}\wedge \omega^{m}~\textrm{on}~{M}. \\[1mm]
\end{array}\right.
\end{equation}
\end{theorem}

Such equations have been well-studied in the past several decades, and it plays an important role in the study of fully nonlinear second order elliptic PDEs.
In the K\"ahler case, when $M$ is compact, the complex \MA~ equation
\begin{equation}\label{Yau}
\chi_{u}^{n}=\psi\omega^{n}
\end{equation}
was studied by Yau in his famous paper \cite{Yau76}, which is a fundamental tool in the study of certain theories of K\"ahler manifolds.  Cao \cite{Cao85} provided a new proof by using the parabolic approach (also called as K\"ahler Ricci flow). To be more precise, he investigated the parabolic \MA~ equation
\begin{equation}\label{X1}
\frac{\partial u_{t}}{\partial t}=\log\frac{\chi_{u_{t}}^{n}}{\omega^{n}}-\log\psi,~ u_{0}=0,
\end{equation}
where we further require $\chi_{u_{t}}>0$. He proved that the solutions of (\ref{X1}) exist for all time and convergence to the solution of (\ref{Yau}).
The  complex Hessian equations
\begin{equation}\label{Yau1}
\chi_{u}^{m}\wedge \omega^{n-m}=\psi\omega^{n}, ~1\leq m\leq n-1,
\end{equation}
were completed by Dinew-Kolodziej \cite{DK17}, where it used the second order estimate of Hou-Ma-Wu \cite{HMW}  to obtain the gradient estimate.

The equation (\ref{dp}) could also be regarded as a natural extension of Donaldson's equation \cite{Don} where $M$ is K\"ahler. Specifically, when $\psi$ is a constant and $m=1$, that is,
\[\chi_{u}^{n}=c\chi_{u}^{n-1}\wedge \omega,\]
where $\chi$ is a smooth closed real $(1,1)$ form on $M$ and $c$ is a constant depending on the classes $[\chi]$ and $[\omega]$.  Chen \cite{Ch00,Ch04} also found it in the study of Mabuchi energy. Further related works see also \cite{FLM11,GLZ09,SW08,Sun3,We04,We06} and references therein. 
For overview of recent development for complex \MA~equations, we refer the reader to the comprehensive survey paper \cite{PSS12}.

If $M$ is a compact Hermitian manifold, Tosatti-Weinkove \cite{TW09} had considered the equation (\ref{Yau}) and gave a complete proof, together with the work of Zhang \cite{Zh10}. The parabolic proof was given by Gill \cite{Gill}. Different from K\"ahler case, when $\chi$ is a smooth real $(1,1)$ form,  Sun \cite{Sun2} also studied equations 
\begin{equation}\label{1...4}
	\chi_{u}^{k}\wedge \omega^{n-k}=\psi\chi_{u}^{l}\wedge \omega^{n-l},~1\leq l<k\leq n
\end{equation}
 with extra $\mathcal{C}$-condition, namely,
\[
k\chi'^{k-1}\wedge \omega^{n-k}>l\psi\chi'^{l-1}\wedge \omega^{n-l}
\]
for some  $k$-positive (with respect to $\omega$) $\chi'\in [\chi]$. More general fully nonlinear second order elliptic equations were investigated by Sz\'{e}kelyhidi \cite{Gab}. The Dirichlet problems
were considered by Guan-Li \cite{GL13} and Guan-Sun \cite{GS13} under the existence of subsolution, which extended the Guan's works \cite{Gu98,Gu14}. One can also refer to \cite{Sun1,Sun2,Sun3,Yu20a,Yu20b} and references therein for recent developments of this topic. The Hessian equations on Hermitian manifolds were considered in \cite{SW19,Zh17} for compact case and by Collins-Picard \cite{CP19} for Dirichlet problem.

The complex \MA~equation  on \acm s was studied by Chu-Tosatti-Weinkove \cite{CTW}. Chu \cite{Chu19} also provided a parabolic proof, analogous to the K\"ahler Ricci flow. Chu-Huang-Zhu \cite{CHZ17} has considered the $\sigma_{2}$ equation and gave the second order estimate. On the strictly pseudoconvex domains of almost complex manifolds, the Dirichlet problem was studied by Harvey-Lawson \cite{HL} for continuous weak solution, and by Plis \cite{P14} for smooth solution by using a similar idea of Caffarelli-Kohn-Nirenberg-Spruck  \cite{CKNS85} on  $\mathbb{C}^{n}$.

Analogous to (\ref{X}), it is also of interest to consider the parabolic version of the equation (\ref{dp}). More precisely, we can consider the following heat equations
\begin{equation}\label{X}
\frac{\partial u_{t}}{\partial t}=\log\frac{\chi_{u_{t}}^{n}}{\chi_{u_{t}}^{n-m}\wedge\omega^{m}}-\log\psi, \qquad u(0,x)\in \mathcal{H}(M),
\end{equation}
where $\chi_{u_{t}}>0$. This result has studied  in the Hermitian case by Sun \cite{Sun4} and  in K\"ahler case by \cite{Ch00,Ch04,FL12,FL13,FLSW14,FLM11,We04,We06} and references therein. We shall prove the existence and convergence of \eqref{X} in elsewhere.

~\\
\textbf{Structure of the paper:} In the \S 2 we will give a brief introduction to the \acm s and reviews some lemmas from \cite{Gab}.

In the \S 3 we will prove the oscillation estimate. As in \cite{CTW}, we need to use the modified
Alexandroff-Bakelman-Pucci maximum principle in \cite{Blo13,Gab}. In the case of \MA~ equation, once we have proved $\chi_{u}\geq \frac{1}{2}\omega$, then the upper bound for $\lambda_{*}(\tilde{g}_{i\bar{j}})$ is easily obtained from a Laplacian inequality. But the Laplacian inequality might not hold for our \MA~ type equations, in order to bound $\lambda_{*}(\tilde{g}_{i\bar{j}})$ from above, we need to further use the condition of $\mathcal{C}$-subsolution.

In the \S 4 we prove the $C^{1}$ estimate by using the maximum principle. We heavily rely on the properties of $\mathcal{C}$-subsolution to derive the Corollary \ref{Cor2.6}. 

In the \S 5 we will give the  $C^{2}$ estimate, which is the most troublesome part in the whole paper. Different from the method of Hou-Ma-Wu \cite{HMW} and G. Sz\'{e}kelyhidi \cite{Gab},  we use the argument of \cite{CTW} together with the $\mathcal{C}$-subsolution to estimate the largest  eigenvalue of the \textit{real} Hessian of $u$. On the lower bound for $\mathcal{L}(Q)$ (see Lemma \ref{Lma4.2}), it is a critical point for us to carefully use the term involving $-G^{i\bar{k},j\bar{l}}$  to control the \textit{bad} negative third order terms.

Once we have proved them, higher order estimates can be also obtained by applying the well-known Evans-Krylov theory (see \cite{Kry,TWWY15} for instance). We omit the proof of the standard steps here. In the last section we will give the proof of existence theorem.

~\\
{\bf Acknowledgments.}
The author wish to thank his advisor professor Xi Zhang for his constant encouragements.  The author is partially supported by NSF in China No.11625106,
11571332 and 11721101. The research was partially supported by the project ``Analysis and Geometry on Bundle" of Ministry of Science and Technology of the People's Republic of China, No. SQ2020YFA070080.
\section{Preliminaries}
On an  almost Hermitian manifold  $(M,g,J)$ with  real dimension $2n$,  for each $(p,q)$-form,    we can define $\partial$ and $\overline{\partial}$ operators  (see \cite{CTW,HL}). Denote $A^{1,1}(M)$ by the space of smooth real (1,1)-forms on $M$. Let $e_{1},\cdots, e_{n}$ be a local $g$-orthonormal frame of $T_{1,0}M$. For $v\in C^{2}(M,\mathbb{R})$, define
\[
g_{i\bar{j}}(v)= g_{i\bar{j}}+e_{i}\bar{e}_{j}v-[e_{i},\bar{e}_{j}]^{0,1}v.
\]
Then in this local frame we have
\[
\chi_{v}=\sqrt{-1}\sum_{i,j} g_{i\bar{j}}(v)\theta_{i}\wedge\bar{\theta}_{j},
\]
where $\theta_{1},\cdots,\theta_{n}$ is the corresponding local $g$-orthonormal coframe of $T^{*}M.$
~\\

In what follows, we will let $\tilde{g}_{i\bar{j}}$,~$\bar{g}_{i\bar{j}}$  be $g_{i\bar{j}}(u)$, $g_{i\bar{j}}(\underline{u})$ respectively for simplicity.
Using this notation, our equation in (\ref{dp}) can be expressed as
\begin{equation}\label{dp'}
F(\tilde{g}_{i\bar{j}})=\left[\frac{\sigma_{n}(\lambda_{*}(\tilde{g}_{i\bar{j}}))}{\sigma_{n-m}(\lambda_{*}(\tilde{g}_{i\bar{j}}))}\right]^{\frac{1}{m}}=h,
\end{equation}
where $h$ is defined by $\psi={n\choose m}h^{m}$, or the inverse Hessian equation 
\begin{equation}\label{dp''}
{G}(\tilde{g}_{i\bar{j}})=-\sigma_{m}(\lambda^{*}(\tilde{g}^{i\bar{j}}))=-{n\choose m}\psi^{-1}=:-\hat{h}.
\end{equation}
Here $\lambda_{*}(A)$ (resp. $\lambda^{*}(A)$) denotes the eigenvalues of Hermitian matrix $A$ with respect to $g$ (resp. $g^{-1}$) and
$\sigma_{k}$ ($1\leq k\leq n$) denotes the $k$-th elementary symmetric polynomial defined by
\[\sigma_{k}(\lambda)=\sum_{1\leq i_{1}<\cdots <i_{k}\leq n}\lambda_{i_{1}}\cdots \lambda_{i_{k}},~\lambda\in \mathbb{R}^{n}.\]
For $\{i_{1},\cdots, i_{s}\}\subset \{1,\cdots n\}$ and each $1\leq k\leq n-1$, let
\[\sigma_{k;i_{1}\cdots i_{s}}(\lambda)=\sigma_{k}(\lambda|_{\lambda_{i_{1}}=\cdots=\lambda_{i_{s}}=0}).\]
We also let
$F^{i\bar{j}}=\frac{\partial F}{\partial{\tilde{g}_{i\bar{j}}}}$, ${G}^{i\bar{j}}=\frac{\partial {G}}{\partial{\tilde{g}_{i\bar{j}}}}$ and
${G}^{i\bar{k},j\bar{l}}=\frac{\partial^{2} {G}}{\partial{\tilde{g}_{i\bar{k}}}\partial{\tilde{g}_{j\bar{l}}}}$ in what follows. It is not hard to see
\[
{G}^{i\bar{j}}=\frac{-mG}{F}F^{i\bar{j}}=\frac{m\hat{h}}{h}F^{i\bar{j}}.
\]
This pro rata means that ${G}$ and $F$ share same certain critical properties in some lemmas below.

Fixed a point $x\in M$, we can choose a local frame around $x$ such that $g_{i\bar{j}}=\delta_{ij}$ and the matrix $\big\{\tilde{g}_{i\bar{j}}\big\}$ is diagonal, then $\big\{{G}^{i\bar{j}}\big\}$ is also diagonal at $x$ and
${G}^{i\bar{i}}=\sigma_{m-1;i}(\tilde{g}^{i\bar{i}})^{2}$ for each $i$.  We denote the (second order) linearization of ${G}$ by
\begin{equation}\label{}
\mathcal{L}=\sum_{i,j}{G}^{i\bar{j}}(e_{i}\bar{e}_{j}-[e_{i},\bar{e}_{j}]^{0,1}).
\end{equation}
\subsection{$\mathcal{C}$-subsolution} 
Denote
\begin{equation*}
	\Gamma_{n}=\{\lambda=(\lambda_{1}, \cdots, \lambda_{n})\in \mathbb{R}^{n},\  \lambda_{i}>0, 1\leq i \leq n\}.
\end{equation*} The following lemma is well-known and one can refer \cite{Mit70} for the proof.
\begin{lemma}\label{Lma2.1} Let $f(\lambda)=\left[\frac{\sigma_{n}(\lambda)}{\sigma_{n-m}(\lambda)}\right]^{\frac{1}{m}}$ defined on $\Gamma_{n}$. Then we have 
\begin{enumerate}
  \item $f_{i}=\frac{\partial f}{\partial \lambda_{i}}>0$ with $f$ is 1-homogeneous and concave;
  \item $f\Big|_{\partial\Gamma_{n} }=0$;
  \item For each $\sigma<\infty$ and $\lambda\in \Gamma_{n}$, we have $\underset{t\ri\infty}{\lim}f(t\lambda)>\sigma$.
\end{enumerate}
\end{lemma}
For any $\sigma>0$, define 
$\Gamma_{n}^{\sigma}=\big\{\lambda\in \Gamma_{n}:~ f(\lambda)>\sigma\big\}.$
We can choose a proper value $\sigma$ such that $\Gamma_{n}^{\sigma}$ is not empty, therefore, $\partial\Gamma_{n}^{\sigma}=f^{-1}(\sigma)$ is a smooth hypersurface in $\Gamma_{n}$.
\begin{defn}\label{subsolution} {\normalfont We say $\underline{u}$ is a $\mathcal{C}$-subsolution for (\ref{dp}) if at each $x\in M$,
\begin{equation}\label{2..11}
(\lambda_{*}(\bar{g}_{i\bar{j}})+\Gamma_{n})\cap \partial\Gamma_{n}^{h(x)}
\end{equation}
is a bounded set.}
\end{defn}
So, there are uniform constants $\delta, R>0$ such that at each $x\in M$ we have
\begin{equation}\label{2..12}
(\lambda_{*}(\bar{g}_{i\bar{j}})-\delta I_{n}+\Gamma_{n})\cap \partial\Gamma_{n}^{h(x)}\subset B_{R}(0),
\end{equation}
where $B_{R}(0)$ is a $R$-radius ball in $\mathbb{R}^{n}$ with center 0.
\begin{remark}
{\normalfont For our equation (\ref{dp}), the set $(\lambda_{*}(\bar{g}_{i\bar{j}})+\Gamma_{n})\cap \partial\Gamma_{n}^{h(x)}$ is bounded means that
\[\lim_{t\ri +\infty}f(\lambda_{*}(\bar{g}_{i\bar{j}})+t\textbf{e}_{i})>h\]
for each $i=1,2,\cdots, n$, where $\textbf{e}_{i}$ is the i-th standard basis vector. }

\end{remark}
Let $\underline{u}\in C^{2}(M,\mathbb{R})$ be an admissible $\mathcal{C}$-subsolution, then there is a uniform constant $0<\tau<1$ such that
\begin{equation}\label{2.12}
\tau^{-1}\omega\geq \chi_{\underline{u}}\geq \tau\omega.
\end{equation}

The following lemma is critical to us (see e.g. \cite{Gab,Gu14}).
\begin{lemma}\label{Lemma2.3}\cite[Proposition 5]{Gab}
Suppose $\mu\subset \mathbb{R}^{n}$ satisfies
\[
(\mu-\delta I_{n}+\Gamma_{n})\cap \partial\Gamma_{n}^{\sigma}\subset B_{R}(0)
\]
for some $\delta, R>0$. Then there exists a constant $\theta>0$ (depending only on $\delta$) such that for $\lambda\in \partial\Gamma_{n}^{\sigma}$ with $|\lambda|>R$, either
\[
\sum_{i}f_{i}(\lambda)(\mu_{i}-\lambda_{i})\geq \theta\sum_{i}f_{i}(\lambda)
\]
or
\[
f_{k}(\lambda)\geq \theta\underset{i}{\sum}f_{i}(\lambda),~ \textrm{for}~ k=1,2,\cdots,n
\]
in $(\mu-\delta I_{n}+\Gamma_{n})\cap \partial\Gamma_{n}^{\sigma}$.
\end{lemma}
There is also a similar result holds for the Hermitian matrices in \cite{Gab,FLM11,Gu14}.
\begin{lemma}\cite[Proposition 6]{Gab}
Let $[a,b]\subset (0,\infty)$ and for some $\delta, R>0$. There exists a constant $\theta>0$ (depending only on $\delta,R$) such that the following holds. Suppose there exist a Hermitian matrix $B$ and a constant $\sigma\in [a,b]$ satisfying
\[
(\lambda(B)-\delta I_{n}+\Gamma_{n})\cap \partial\Gamma_{n}^{\sigma}\subset B_{R}(0).
\]
Then for $\lambda(A)\in (\lambda(B)-\delta I_{n}+\Gamma_{n})\cap \partial\Gamma_{n}^{\sigma}$, we have either
\[
\sum_{i,j}F^{ij}(A)(B_{ij}-H_{ij})\geq \theta\sum_{i}F^{ii}(A)
\]
or
\[
F^{kk}(A)\geq \theta\underset{i}{\sum}F^{ii}(A),~  \textrm{for}~ k=1,2,\cdots,n.
\]
\end{lemma}
\begin{proof}
When $|\lambda(A)|>R$, the conclusion follows from \cite[Proposition 6]{Gab}. So we may assume $|\lambda(A)|\leq R$, consider the set
		\[
		S_{R,\sigma} = \left\{N\in \mathcal{H}:\mu(N)\in \overline{B_{R}(0)}\cap\partial \Gamma^{\sigma}\right\},
		\]
		where $\mathcal{H}$ denotes the set of Hermitian matrices.
	It is clear that $S_{R,\sigma}$ is compact, and then there exists a constant $C>0$ such that for $A\in S_{R,\sigma}$,
	\[
	C^{-1} \leq F^{k\bar{k}}(A) \leq C, \quad \text{for} \ k=1,2,\cdots,n.
	\]
	Decreasing $\theta$ if necessary,
	\[
	F^{k\bar{k}}(A)>\theta\sum_{p}F^{p\bar{p}}(A), \quad \text{for} \ k=1,2,\cdots,n.
	\]
\end{proof}

For the  $f$ defined in (\ref{dp'}), the following result is also important for us.
\begin{lemma}\label{Lma2.5}\cite{Gab,FLM11,Gu14,Spr}
For each $\sigma\in (0, \infty)$.
\begin{enumerate}
  \item There is a large constant $N_{0}$ depending on $\sigma$, such that for all $N\geq N_{0}$,
      \[
        \Gamma_{n}+NI_{n}\subset\Gamma_{n}^{\sigma}.
      \]
  \item There is a positive constant $\kappa$ depending on $\sigma$ such that
  \[
  \sum_{i}f_{i}(\lambda)>\kappa
  \]
  for any $\lambda\in\partial\Gamma_{n}^{\sigma}$.
\end{enumerate}
\end{lemma}
We will use the following corollary frequently in the next sections.
\begin{corollary}\label{Cor2.6}
Assume $\underline{u}$ is an admissible $\mathcal{C}$-subsolution and $u$ is the solution for (\ref{dp}). Then
there exists a constant $\theta>0$ (depending only on $\delta$) such that if $|\lambda_{*}(\tilde{g}_{i\bar{j}})|>R$ for some constant $R$ (depending on the allowed data) such that \[
(\lambda_{*}(\bar{g}_{i\bar{j}})-\delta I_{n}+\Gamma_{n})\cap \partial\Gamma_{n}^{h}\subset B_{R}(0).
\] Then either
\begin{equation}\label{2..25}
\mathcal{L}(\underline{u}-u)=\sum_{i,j}{G}^{i\bar{j}}(\tilde{g})(\bar{g}_{i\bar{j}}-\tilde{g}_{i\bar{j}})\geq \theta\sum_{i}{G}^{i\bar{i}}(\tilde{g})
\end{equation}
or
\begin{equation}\label{2..26}
{G}^{k\bar{k}}(\tilde{g})\geq \theta\underset{i}{\sum}{G}^{i\bar{i}}(\tilde{g}),~ \textrm{for}~ k=1,2,\cdots,n
\end{equation}
 if  $\lambda_{*}(\tilde{g}_{i\bar{j}})
 \in(\lambda_{*}(\bar{g}_{i\bar{j}})-\delta I_{n}+\Gamma_{n})\cap \partial\Gamma_{n}^{h}$.

In addition, there is a constant $\Theta>0$ depending on $h$ such that
  \begin{equation}\label{2..27}
  \mathcal{G}=\sum_{i}{G}^{i\bar{i}}(\tilde{g})>\Theta,~ \textrm{if}~\lambda_{*}(\tilde{g}_{i\bar{j}})\in\partial\Gamma_{n}^{h}.
  \end{equation}

\end{corollary}
\section{oscillation estimate}
The following $L^{1}$-estimate is well-known, for instance, see \cite{CTW} et al.
\begin{proposition}\label{Prop3.1}\cite[Proposition 2.3]{CTW}
There exists a constant $C_{0}>0$ depending only on $(M,\omega, J)$ such that for every smooth real function $u\in \mathcal{H}(M)$ with $\sup_{M}(u-\underline{u})=0$. Then
\begin{equation}\label{}
\int_{M}|u-\underline{u}|\omega^{n}\leq C_{0}.
\end{equation}
\end{proposition}
The main result of this section is the following proposition.
\begin{proposition}\label{Prop3.2}
Let $u$ be the solution for (\ref{dp}) with $\sup_{M}(u-\underline{u})=0$. Then
\begin{equation}\label{}
osc_{M}(u-\underline{u})\leq C
\end{equation}
for some constant $C>0$ depending on the allowed data.
\end{proposition}
\begin{proof}
From the hypothesis, it suffices to estimate the infimum $m_{0}=\inf_{M}(u-\underline{u})$, we may assume $m_{0}$ is attained at the origin. Choose a local coordinate chart $\{x^{1},\cdots, x^{2n}\}$ in a neighborhood of the origin containing the unit ball $B_{1}\subset \mathbb{R}^{2n}$.

Consider the test function
\[
v= u-\underline{u}+\ve\sum_{i=1}^{2n}(x^{i})^{2}
\]
for a small $\ve>0$ determined later. Then we have
\[v(0)=m_{0}~ \textrm{and}~v\geq m_{0}+\ve ~\textrm{on}~ \partial B_{1}.\]
We define the lower contact set of $v$ by
\[
\mathcal{S}=\Big\{x\in B_{1}: |Dv(x)|\leq \frac{\ve}{2},
v(y)\geq v(x)+Dv(x)\cdot (y-x),\forall y\in B_{1}\Big\}.
\]
~\\
By the modified Alexandroff-Bakelman-Pucci maximum principle (see \cite{Blo13} or \cite[Proposition 10]{Gab} for instance), there is  constant $c_{0}=c_{0}(n)>0$  such that
\begin{equation}\label{3.5}
c_{0}\ve^{2n}\leq \int_{\mathcal{S}}\det(D^{2}v).
\end{equation}
In what follows, the constant $C$ below in this section may be changed from line to line, but it depends on the allowed data.
Note that $0\in \mathcal{S}$ and $D^{2}v\geq 0$ on $\mathcal{S}$. Moreover,
\[
\big(D^{2}(u-\underline{u})\big)^{J}(x)\geq (D^{2}v)^{J}(x)-C\ve Id\geq -C\ve Id,
\]
and using the fact $|D(u-\underline{u})|\leq \frac{5\ve}{2}$ on $\mathcal{S}$ by definition, which implies \[
\begin{split}
H(u)-H(\underline{u})=&\big(D^{2}u+E(Du)\big)^{J}-\big(D^{2}\underline{u}
+E(D\underline{u})\big)^{J}  \\
    =&\big(D^{2}(u-\underline{u})\big)^{J}+ \big(E(D(u-\underline{u}))\big)^{J}\\
    \geq& -C\ve Id.
\end{split}
\]
Hence
$\chi_{u}-\chi_{\underline{u}}\geq -C\ve\omega.$
Therefore, if we choose $\ve$ sufficient small such that $C\ve\leq \delta$, then
$\lambda_{*}(\tilde{g}_{i\bar{j}})\in \lambda_{*}(\bar{g}_{i\bar{j}})-\delta I_{n}+\Gamma_{n}.$
On the other hand, the equation (\ref{dp'}) implies $\lambda_{*}(\tilde{g}_{i\bar{j}})\in \partial\Gamma_{n}^{h(x)}.$
Consequently, \[\lambda_{*}(\tilde{g}_{i\bar{j}})\in (\lambda_{*}(\bar{g}_{i\bar{j}})-\delta I_{n}+\Gamma_{n})\cap \partial\Gamma_{n}^{h(x)} \subset B_{R}(0)\] for some $R>0$ as (\ref{2..12}). This gives an upper bound for $H(u)$ and hence also for $H(v)-E(Dv)$ on $\mathcal{S}$. The rest of the proof is analogous to \cite{CTW}, we present it on the below for completeness .

Using the fact $\det(A+B)\geq \det(A)+\det(B)$ for positive definite Hermitian matrices $A, B$, on $\mathcal{S}$, we have
\begin{equation*}\label{3.10}
\begin{split}
\det(D^{2}v)\leq & 2^{2n-1}\det((D^{2}v)^{J}) \\
              =  & 2^{2n-1}\det(H(v)-E(Dv))\leq C.
\end{split}
\end{equation*}
Plugging it into (\ref{3.5}) gives us
\[
c_{0}\ve^{2n}\leq C|\mathcal{S}|.
\]
For each $x\in \mathcal{S}$,  we have $D^{2}v(x)\geq 0$, then
\[
m_{0}=v(0)\geq v(x)-|Dv(x)||x|\geq v(x)-\frac{\ve}{2}.
\]
We may assume $m_{0}+\ve< 0$ (otherwise we are done), then, on $\mathcal{S}$,
\[
-v\geq |m_{0}+\ve|.
\]
It follows from Proposition \ref{Prop3.1} that
\[
c_{0}\ve^{2n}\leq C|\mathcal{S}|\leq C\frac{\int_{\mathcal{S}}(-v)\omega^{n}}{|m_{0}+\ve|}\leq \frac{C}{|m_{0}+\ve|}.
\]
Then this gives a uniform lower bound for $m_{0}$.
\end{proof}
\section{First order estimate}
\begin{proposition}
Let $u$ (resp. $\underline{u}$) be the solution (resp. $\mathcal{C}$-subsolution) for (\ref{dp}) with $\sup_{M}(u-\underline{u})=0$. Then
\begin{equation}\label{}
|\partial u|\leq C_{0}
\end{equation}
for some positive constant $C_{0}$ depending on the allowed data and $\underline{u}$.
\end{proposition}
~\\ \textit{Proof.}
Let $w= Ae^{B\eta}$ for $\eta=\underline{u}-u-\inf_{M}(\underline{u}-u),$ where $A, B>0$\footnote{The $C, C_{0}$ below in this section denote the constants that may change from line to line, where $C_{0}$ depends on all the allowed data, but $C$ does not depend on $A, B$ that we are yet to choose.} are certain constants to be chosen soon. Consider the function
\[V= e^{w}|\partial u|^{2}.\]
Suppose $V$ achieves maximum at the origin. Then near the origin, we can choose a  local $g$-unitary frame still denoted by $e_{1},\cdots,e_{n}$ such that $g_{i\bar{j}}=\delta_{ij}$ and the matrix $\big\{\tilde{g}_{i\bar{j}}\big\}$ is diagonal.


By the maximum principle, at the origin, it follows that
\begin{equation}\label{3.9}
\begin{split}
  0\geq \frac{\mathcal{L}(V)}{Bw e^{w}|\partial u|^{2}}=&\frac{\mathcal{L}(e^{w})}{Bwe^{w}}+\frac{\mathcal{L}(|\partial u|^{2})}{Bw|\partial u|^{2}}+2{G}^{i\bar{i}}\textrm{Re}\Big\{e_{i}({w})
  \frac{\bar{e}_{i}(|\partial u|^{2})}{Bw|\partial u|^{2}}\Big\}\\
  =&\mathcal{L}(\eta)+B(1+w){G}^{i\bar{i}}|\eta_{i}|^{2}+\frac{\mathcal{L}(|\partial u|^{2})}{Bw|\partial u|^{2}}\\
  &+\frac{2}{|\partial u|^{2}}\sum_{j}{G}^{i\bar{i}}\textrm{Re}\big\{e_{i}(\eta)\bar{e}_{i}e_{j}(u)\bar{e}_{j}(u)
    +e_{i}(\eta)\bar{e}_{i}\bar{e}_{j}(u)e_{j}(u)\big\}.
\end{split}
\end{equation}
By direct calculation,
\begin{equation}\label{Lpartaiu}
\mathcal{L}(|\partial u|^{2})={G}^{i\bar{i}}\Big({e_{i}e_{\bar{i}}}(|\partial u|^{2})-[e_{i},\bar{e}_{i}]^{0,1}(|\partial u|^{2})\Big)= I+II+III,	
\end{equation}
where
\[I={G}^{i\bar{i}}(e_{i}\bar{e}_{i}e_{j}u-[e_{i},\bar{e}_{i}]^{0,1}e_{j}u)\bar{e}_{j}u;\]
\[II={G}^{i\bar{i}}(e_{i}\bar{e}_{i}\bar{e}_{j}u-[e_{i},\bar{e}_{i}]^{0,1}\bar{e}_{j}u)e_{j}u;\]
\[III={G}^{i\bar{i}}(|e_{i}e_{j}u|^{2}+|e_{i}\bar{e}_{j}u|^{2}).\]
Differentiating (\ref{dp''}) once, since $g$ is almost Hermitian, we have
\[ {G}^{i\bar{i}}(e_{j}e_{i}\bar{e}_{i}u-e_{j}[e_{i},\bar{e}_{i}]^{0,1}u)=-\hat{h}_{j}.
\]
Note that
\[
\begin{split}
    &  {G}^{i\bar{i}}(e_{i}\bar{e}_{i}e_{j}u-[e_{i},\bar{e}_{i}]^{0,1}e_{j}u) \\
   = &  {G}^{i\bar{i}}(e_{j}e_{i}\bar{e}_{i}u+e_{i}[\bar{e}_{i},e_{j}]u
   +[e_{i},e_{j}]\bar{e}_{i}u-[e_{i},\bar{e}_{i}]^{0,1}e_{j}u)\\
   = & -\hat{h}_{j}+{G}^{i\bar{i}}e_{j}[e_{i},\bar{e}_{i}]^{0,1}u
   +{G}^{i\bar{i}}(e_{i}[\bar{e}_{i},e_{j}]u
   +[e_{i},e_{j}]\bar{e}_{i}u-[e_{i},\bar{e}_{i}]^{0,1}e_{j}u)\\
   =&-\hat{h}_{j}+ {G}^{i\bar{i}}\big\{e_{i}[\bar{e}_{i},e_{j}]u+\bar{e}_{i}[e_{i},e_{j}]u+[[e_{i},e_{j}],\bar{e}_{i}]u
   -[[e_{i},\bar{e}_{i}]^{0,1},e_{j}]u\big\}
.\end{split}
\]
We may and do assume $|\partial u|\gg 1$. Therefore,
\begin{equation}\label{I+II}
\begin{split}
I+II\geq& -2\sum_{j}\textrm{Re}\big\{\hat{h}_{j}u_{\bar{j}}\big\}-C|\partial u|\sum_{j}{G}^{i\bar{i}}(|e_{i}e_{j}u|+|e_{i}\bar{e}_{j}u|)-C|\partial u|^{2} \mathcal{G}\\
   \geq &-2\sum_{j}\textrm{Re}\big\{\hat{h}_{j}u_{\bar{j}}\big\}-\frac{C}{\varepsilon}|\partial u|^{2}\mathcal{G}-\ve\sum_{j} {G}^{i\bar{i}}(|e_{i}e_{j}u|^{2}+|e_{i}\bar{e}_{j}u|^{2}).
\end{split}
\end{equation}
Plugging \eqref{I+II} into \eqref{Lpartaiu},
\begin{equation}\label{3.7}
\begin{split}
\frac{\mathcal{L}(|\partial u|^{2})}{Bw|\partial u|^{2}}\geq& -\frac{2\sum_{j}\textrm{Re}\big\{\hat{h}_{j}u_{\bar{j}}\big\}}{Bw|\partial u|^{2}}+ (1-\varepsilon) \sum_{j}{G}^{i\bar{i}}\frac{|e_{i}e_{j}u|^{2}+|e_{i}\bar{e}_{j}u|^{2}}{Bw|\partial u|^{2}}- \frac{C\mathcal{G}}{Bw\varepsilon} .
\end{split}
\end{equation}
Now we estimate the last term of (\ref{3.9}). On the one hand, using Cauchy-Schwarz (we will use C-S for convenience) inequality, for $0<\ve\leq \frac{1}{2}$ we have
\[
\begin{split}
&2\sum_{j}{G}^{i\bar{i}}\textrm{Re}\big\{e_{i}(\eta)
\bar{e}_{i}e_{j}(u)\bar{e}_{j}(u)\big\}\\
=& 2\sum_{j}{G}^{i\bar{i}}\textrm{Re}\big\{\eta_{i}u_{\bar{i}}\big\{e_{j}\bar{e}_{i}(u)-[e_{j},\bar{e}_{i}]^{0,1}(u)
-[e_{j},\bar{e}_{i}]^{1,0}(u)\big\}\big\}\\
=&  2{G}^{i\bar{i}}(\tilde{g}_{i\bar{i}}-1)\textrm{Re}\big\{\eta_{i}u_{\bar{i}}\big\}-2\sum_{j}{G}^{i\bar{i}}\textrm{Re}\big\{e_{i}(\eta)
\bar{e}_{j}(u)[e_{j},\bar{e}_{i}]^{1,0}(u)\big\}\\
\geq&  2{G}^{i\bar{i}}\tilde{g}_{i\bar{i}}\textrm{Re}\{\eta_{i}u_{\bar{i}}\}-\varepsilon Bw|\partial u|^{2}{G}^{i\bar{i}}|\eta_{i}|^{2}
-\frac{C}{Bw\varepsilon}|\partial u|^{2}\mathcal{G}.
\end{split}
\]
On the other hand, since for $0<\varepsilon\leq\frac{1}{2}$, then $1\leq (1-\varepsilon)(1+2\varepsilon)$. By using the C-S inequality again we achieve
\[
\begin{split}
   &2\sum_{j}{G}^{i\bar{i}}\textrm{Re}\big\{e_{i}(\eta)\bar{e}_{i}\bar{e}_{j}(u)e_{j}(u)\big\} \\
\geq & -\frac{(1-\varepsilon)}{Bw}\sum_{j}{G}^{i\bar{i}}|\bar{e}_{i}\bar{e}_{j}(u)|^{2}
-(1+2\varepsilon)Bw|\partial u|^{2}{G}^{i\bar{i}}|\eta_{i}|^{2}.
\end{split}
\]
Therefore,
\begin{equation}\label{3.12}
\begin{split}
   2{G}^{i\bar{i}}\textrm{Re}\Big\{e_{i}({w})\frac{\bar{e}_{i}(|\partial u|^{2})}{Bw|\partial u|^{2}}\Big\}
    \geq &2{G}^{i\bar{i}}\tilde{g}_{i\bar{i}}\frac{\textrm{Re}\{\eta_{i}u_{\bar{i}}\}}{|\partial u|^{2}}-(1+3\ve)Bw{G}^{i\bar{i}}|\eta_{i}|^{2}
\\&-\frac{C}{Bw\varepsilon}\mathcal{G}
-(1-\varepsilon)\sum_{j}{G}^{i\bar{i}}\frac{|\bar{e}_{i}\bar{e}_{j}(u)|^{2}}{Bw|\partial u|^{2}}.
\end{split}
\end{equation}
Combining  (\ref{3.9}), (\ref{3.7})-(\ref{3.12}), we obtain
\[
\begin{split}
   0 \geq & \mathcal{L}(\eta)+B(1+w){G}^{i\bar{i}}|\eta_{i}|^{2}-\frac{2C}{Bw\varepsilon}\mathcal{G}
 -\frac{2\sum_{j}\textrm{Re}\{\hat{h}_{j}u_{\bar{j}}\}}{Bw|\partial u|^{2}}\\
    &+2{G}^{i\bar{i}}\tilde{g}_{i\bar{i}}\frac{\textrm{Re}\{\eta_{i}u_{\bar{i}}\}}{|\partial u|^{2}}-(1+3\ve)Bw{G}^{i\bar{i}}{|\eta_{i}|^{2}}\\
\geq &  \mathcal{L}(\eta)+B(1-3\ve w) {G}^{i\bar{i}}|\eta_{i}|^{2}-\frac{2C}{Bw\varepsilon}\mathcal{G}
-\frac{C}{Bw|\partial u|}
 +2{G}^{i\bar{i}}\tilde{g}_{i\bar{i}}\frac{\textrm{Re}\{\eta_{i}u_{\bar{i}}\}}{|\partial u|^{2}}.
\end{split}
\]
Hence, choose $\ve=\frac{1}{6w(0)}$ ($\leq \frac{1}{2}$ if $A$ is large enough), this gives us
\begin{equation}\label{4..16}
\begin{split}
&\mathcal{L}(\eta)+ {G}^{i\bar{i}}\tilde{g}_{i\bar{i}}\frac{2\textrm{Re}\{\eta_{i}u_{\bar{i}}\}}{|\partial u|^{2}}
+\frac{B}{2}
{G}^{i\bar{i}}|\eta_{i}|^{2}
\leq \frac{C}{Bw|\partial u|}+ \frac{C}{B}\mathcal{G}.
\end{split}
\end{equation} 
~\\
\textbf{Case 1.} If $|\lambda_{*}(\tilde{g}_{i\bar{j}})|\geq R$ for some positive constant \textit{R} large as in Corollary \ref{Cor2.6}.
~\\
(a). First, suppose (\ref{2..25}) holds, we divide the proof into two parts.
~\\
(a-1). Assume ${G}^{j\bar{j}}\geq D$ for some $j$, where $D>0$ is a constant to be chosen soon. Therefore,
\[
\mathcal{L}(\eta)\geq\theta\mathcal{G}\geq \frac{D\theta}{2}+\frac{\theta}{2}\mathcal{G}.
\]
Now we may assume $|\partial u|\geq |\partial\underline{u}|$, then
\[
|\partial \eta|\leq |\partial (\underline{u}-u)|\leq 2|\partial u|.
\]
Hence
\begin{equation}\label{4..21}
{G}^{i\bar{i}}\tilde{g}_{i\bar{i}}\frac{2\textrm{Re}\{\eta_{i}u_{\bar{i}}\}}{|\partial u|^{2}}\geq -4{G}^{i\bar{i}}\tilde{g}_{i\bar{i}}=-4m\hat{h}.
\end{equation}
Plugging it into (\ref{4..16}) gives us
\begin{equation*}\label{}
\frac{D\theta}{2}-4m\hat{h}+\Big(\frac{\theta}{2}-\frac{C}{B}\Big)\mathcal{G}\leq  \frac{C}{Bw|\partial u|}.
\end{equation*}
~\\
Choosing $B, D$ sufficiently large such that the third term  above can be cancelled and $\frac{D\theta}{2}\geq 1+4m\sup_{M}\hat{h}$. Therefore, $|\partial u|\leq C_{0}$.
~\\
(a-2). Assume ${G}^{j\bar{j}}\leq D$ for each $j$. We may assume $|\partial u|\geq \max\{1,|\partial\underline{u}|\}$ and use the C-S inequality to obtain
\[
\begin{split}
{G}^{i\bar{i}}\tilde{g}_{i\bar{i}}\frac{2\textrm{Re}\{\eta_{i}u_{\bar{i}}\}}{|\partial u|^{2}} \geq    -\frac{B}{4}{G}^{i\bar{i}}|\eta_{i}|^{2}
-\frac{4}{B|\partial u|^{2}}{G}^{i\bar{i}}\tilde{g}_{i\bar{i}}^{2}.
\end{split}
\]
It follows that
\[\frac{1}{2}\theta\Theta
+\frac{\theta}{2}\mathcal{G}\leq  \frac{C}{Bw|\partial u|}+ \frac{C}{B}\mathcal{G}+\frac{4}{B|\partial u|^{2}}{G}^{i\bar{i}}\tilde{g}_{i\bar{i}}^{2}. \]
For the choice of $B$ as before, we have
\[
\frac{1}{2}\theta\Theta\leq \frac{C}{Bw|\partial u|}+\frac{4}{B|\partial u|^{2}}{G}^{i\bar{i}}\tilde{g}_{i\bar{i}}^{2}.
\]
It suffices to prove for each $i$,
\begin{equation}\label{4.9}
  \sigma_{m-1,i}(\lambda^{*}(\tilde{g}^{i\bar{j}}))={G}^{i\bar{i}}\tilde{g}_{i\bar{i}}^{2}\leq C.
\end{equation}
We may assume $\tilde{g}^{1\bar{1}}\geq \cdots\geq \tilde{g}^{n\bar{n}}$ at the origin. Hence
\[
\prod_{i=1}^{m}\tilde{g}^{i\bar{i}}\geq \frac{\sigma_{m}(\lambda^{*}(\tilde{g}^{i\bar{j}}))}{{n\choose m}}=\psi^{-1}.
\]
Therefore,
\begin{equation}\label{3.13}
\psi^{-1}\tilde{g}^{1\bar{1}}\leq (\tilde{g}^{1\bar{1}})^{2}\prod_{i=2}^{m}\tilde{g}^{i\bar{i}}\leq \sigma_{m-1;1}\cdot (\tilde{g}^{1\bar{1}})^{2}={G}^{1\bar{1}}\leq D.
\end{equation}
It follows maximality of $\tilde{g}^{1\bar{1}}$ among all the $\{\tilde{g}^{i\bar{i}}\}$, we have
\[
\sigma_{m-1;i}(\lambda^{*}(\tilde{g}^{i\bar{j}}))\leq C(\tilde{g}^{1\bar{1}})^{m-1}\leq C(D\psi)^{m-1}\leq C_{0}
\]
and this proves (\ref{4.9}).
~\\
(b). Second, suppose (\ref{2..26}) holds, then we have
\begin{equation}\label{4...15}
{G}^{k\bar{k}}\geq \theta\mathcal{G}\geq \theta\Theta, ~ \textrm{for each}~k.
\end{equation}
Using (\ref{2.12}) gives us
\begin{equation}\label{4.31}
\mathcal{L}(\eta)={G}^{i\bar{i}}(\bar{g}_{i\bar{i}}-\tilde{g}_{i\bar{i}})\geq \tau\mathcal{G}-{G}^{i\bar{i}}\tilde{g}_{i\bar{i}}.
\end{equation}
Plugging (\ref{4..21}) and (\ref{4.31}) into (\ref{4..16}) implies
\[(\tau-\frac{C}{B}) \mathcal{G}+\frac{B}{2}
{G}^{i\bar{i}}|\eta_{i}|^{2}\leq 5{G}^{i\bar{i}}\tilde{g}_{i\bar{i}}+\frac{C}{Bw|\partial \eta|}
  \leq   C+\frac{C}{Bw|\partial \eta|}.\]
Choosing $B$ large enough once again such that $B\tau\geq C$, so the first term on the left hand can be cancelled. Hence by (\ref{4...15}) we have
\[
\frac{B\theta\Theta}{2}|\partial\eta|^{2}\leq
C+\frac{C}{Bw|\partial \eta|},
\]
which gives an upper bound for $|\partial \eta|$ and hence $|\partial u|\leq C_{0}$.
~\\
\textbf{Case 2: }If $|\lambda_{*}(\tilde{g}_{i\bar{j}})|\leq R$, then we have $\tilde{g}^{k\bar{k}}\geq (CR)^{-1}$ for each $k$.
Therefore,
\[
{G}^{i\bar{i}}|\eta_{i}|^{2}\geq \frac{|\partial \eta|^{2}}{CR^{m+1}}.
\]
Plugging into (\ref{4..16}) it follows
\[
(\tau-\frac{C}{B}) \mathcal{G}+
\frac{B|\partial \eta|^{2}}{CR^{m+1}}\leq \frac{C}{Bw|\partial \eta|}+5m\hat{h}
\]
since $|\partial u|\geq \max\{1,|\partial \underline{u}|\}$. Then by the choice of $B$ as before we have
\[
\frac{B|\partial \eta|^{2}}{CR^{m+1}}\leq \frac{C}{Bw|\partial \eta|}+5m\sup_{M}\hat{h}.
\]
This implies $|\partial\eta|\leq C_{0}$ and hence
$
|\partial u|\leq |\partial \eta|+|\partial\underline{u}|\leq C_{0}.$\qed

\section{Second order estimate}
\begin{theorem}\label{Thm4.1}
Let $u$ (resp. $\underline{u}$) be the solution (resp. $\mathcal{C}$-subsolution) for (\ref{dp}) with $\sup_{M}(u-\underline{u})=0$. Then there exists a constant $C_{0}>0$ depends on the allowed data (including $\|\underline{u}\|_{C^{2}(M)}$ and $\|u\|_{C^{1}(M)}$) such that
\begin{equation}\label{aa}
\|\nabla^{2}u\|_{C^{0}(M)}\leq C_{0}.
\end{equation}
\end{theorem}
Analogous to the arguments in \cite{CTW}, it suffices to bound the largest eigenvalue $\lambda_{1}(\nabla^{2}u)$ of the \textit{real} Hessian $\nabla^{2}u$ with respect to $g$ from above. First, we consider the test function
\[
\mathcal{Q}= \log \lambda_{1}(\nabla^{2}u)+\phi(|\partial u|^{2})+\varphi(\widetilde{\eta})
\]
on $\Omega=\{\lambda_{1}(\nabla^{2}u)>0\}\subset M$. Here $\varphi$ is a function defined by
\[
\varphi(\widetilde{\eta})= e^{B\widetilde{\eta}},
~~\widetilde{\eta}=\underline{u}-u+\sup_{M}(u-\underline{u})+1
\]
for a real constant $B>0$ to be determined later, and $\phi$ is defined by
\[
\phi(s)=-\frac{1}{2}\log(1+\sup_{M}|\partial u|^{2}-s).
\]
Set $K=1+\sup_{M}|\partial u|^{2}$. Note that
\[
   \frac{1}{2K}\leq \phi'(|\partial u|^{2})\leq \frac{1}{2},~~\phi''=2(\phi')^{2}.
\]
We may assume $\Omega$ is a nonempty (relative) open set (otherwise we are done). When $z$ approaches to $\partial{\Omega}$, then $\mathcal{Q}(z)\ri -\infty$.
Suppose $\mathcal{Q}$ achieves a maximum at the origin in $\Omega$ (after a translation), we can choose a proper local frame $x^{1}, \cdots, x^{2n}$ as in \cite{CTW} such that
\begin{equation}\label{5..6}
g_{i\bar{j}}=\delta_{i\bar{j}},~ \partial_{
}g_{\alpha\beta}=0~ \textrm{and the matrix}~ \big\{\tilde{g}_{i\bar{j}}\big\} ~\textrm{is diagonal at the origin}.
\end{equation}
We may further assume at the origin, $\tilde{g}_{1\bar{1}}\geq \cdots\geq \tilde{g}_{n\bar{n}}.$

Consider the perturbation $\Phi$ for $\nabla^{2}u$
 defined by
$
\Phi_{\beta}^{\alpha}=\sum_{\gamma}g^{\alpha\gamma}\{\nabla_{\gamma\beta}^{2}u-S_{\gamma\beta}\}$
for some smooth section $S$ on $T^{*}M\otimes T^{*}M$ such that
$\lambda_{1}(\Phi)\leq \lambda_{1}(\nabla^{2}u)~ \textrm{on} ~\Omega$
with equality only at the origin, but also $\lambda_{1}(\Phi)\in C^{2}(\Omega)$ (see e.g. \cite{CTW,Gab}). Let $V_{1},\cdots, V_{2n}$ be the eigenvectors for $\Phi$ at the origin with eigenvalues $\lambda_{1}(\Phi),\cdots, \lambda_{2n}(\Phi)$ respectively.  We (locally) define the test function
\[
Q=\log(\lambda_{1}(\Phi))+\phi(|\partial u|^{2})+\varphi(\widetilde{\eta}).
\]
For simplicity, we denote $\lambda_{\beta}=\lambda_{\beta}(\Phi)$.

In what follows, we will use the Einstein summation convention, and all the following calculations are done at the origin. The $C, C_{0}$ below in this section denotes positive constants those may change from line to line, where $C_{0}$ depends all the allowed data, but $C$ does not depend on $B$ that we are yet to choose.

By the maximum principle, at the origin, for each $i=1,2,\cdots,n$, we have
\begin{equation}\label{5..10}
\frac{1}{\lambda_{1}}e_{i}(\lambda_{1})=-\phi'e_{i}(|\partial u|^{2})-Be^{B\widetilde{\eta}}e_{i}(\widetilde{\eta}),
\end{equation}
\begin{equation}\label{}
\begin{split}
0\geq \mathcal{L}(Q)=& \frac{\mathcal{L}(\lambda_{1})}{\lambda_{1}}- {G}^{i\bar{i}}\frac{|e_{i}(\lambda_{1})|^{2}}{\lambda_{1}^{2}}  +\phi'' {G}^{i\bar{i}}|e_{i}(|\partial u|^{2})|^{2} \\
      &+\phi'\mathcal{L}(|\partial u|^{2})+Be^{B\widetilde{\eta}}\mathcal{L}(\widetilde{\eta})+B^{2}e^{B\widetilde{\eta}} {G}^{i\bar{i}}|e_{i}(\widetilde{\eta})|^{2}.
\end{split}
\end{equation}
In this section we may and do assume $|\lambda_{*}(\tilde{g}_{i\bar{j}})|\geq R$ for some positive constant $R$ such that Corollary \ref{Cor2.6} holds true (otherwise we are done).
\subsection{Lower bound for $\mathcal{L}(Q)$}
The main result of this subsection is the following lemma.
\begin{lemma}\label{Lma4.2}
For each $\ve\in (0,\frac{1}{2}]$, at the origin, we have
\begin{equation}\label{4.7"}
\begin{split}
\mathcal{L}(Q) \geq & (2-\ve)\sum_{\alpha>1}{G}^{i\bar{i}}\frac{|e_{i}(u_{V_{\alpha}V_{1}})|^{2}}{\lambda_{1}(\lambda_{1}-\lambda_{\alpha})}
-\frac{1}{\lambda_{1}} {G}^{i\bar{k},j\bar{l}}V_{1}(\tilde{g}_{i\bar{k}})V_{1}(\tilde{g}_{j\bar{l}})\\&
-(1+\ve){G}^{i\bar{i}}\frac{|e_{i}(\lambda_{1})|^{2}}{\lambda_{1}^{2}}
-\frac{C}{\ve}\mathcal{G}
+ \frac{\phi'}{2}\sum_{j} {G}^{i\bar{i}}(|e_{i}e_{j}u|^{2}+|e_{i}\bar{e}_{j}u|^{2})\\&+\phi'' {G}^{i\bar{i}}|e_{i}(|\partial u|^{2})|^{2}+Be^{B\widetilde{\eta}}\mathcal{L}(\widetilde{\eta})+B^{2}e^{B\widetilde{\eta}} {G}^{i\bar{i}}|e_{i}(\widetilde{\eta})|^{2}.\\
\end{split}
\end{equation}
\end{lemma}
 First, we calculate $\mathcal{L}(\lambda_{1})$. Let $u_{ij}=e_{i}e_{j}u-(\nabla_{e_{i}}e_{j})u $ and $u_{V_{i}V_{j}}=u_{kl}V^{k}_{i}V^{l}_{j}$. The first and second derivative of $\lambda_{1}$ can be found in \cite{CTW,Gab,Spr} et al, we have
\begin{equation}\label{5..13}
\begin{split}
\mathcal{L}(\lambda_{1})\geq & 2\sum_{\alpha>1}{G}^{i\bar{i}}\frac{|e_{i}(u_{V_{\alpha}V_{1}})|^{2}}{\lambda_{1}-\lambda_{\alpha}}
+ {G}^{i\bar{i}}(e_{i}\bar{e}_{i}-[e_{i},\bar{e}_{i}]^{0,1})(u_{V_{1}V_{1}})-C\lambda_{1}\mathcal{G}.\\
\end{split}
\end{equation}
Since $g$ is almost Hermitian, it follows that
\[ {G}^{i\bar{i}}V_{1}\Big\{e_{i}\bar{e}_{i}(u)-[e_{i},\bar{e}_{i}]^{0,1}(u)\Big\}=-V_{1}(\hat{h}).
\]
Differentiating with $V_{1}$ again, we obtain
\begin{equation}\label{5..15}
{G}^{i\bar{i}}V_{1}V_{1}(\tilde{g}_{i\bar{i}})=- {G}^{i\bar{k},j\bar{l}}V_{1}(\tilde{g}_{i\bar{k}})V_{1}(\tilde{g}_{j\bar{l}})-V_{1}V_{1}(\hat{h}).
\end{equation}
%
%
%
%
~\\
\textbf{Claim 1.} If $\lambda_{1}\gg 1$, then
\begin{equation}\label{claim2}
\begin{split}
 {G}^{i\bar{i}}(e_{i}\bar{e}_{i}-[e_{i},\bar{e}_{i}]^{0,1})(\lambda_{1})
 \geq   & - {G}^{i\bar{k},j\bar{l}}V_{1}(\tilde{g}_{i\bar{k}})V_{1}(\tilde{g}_{j\bar{l}})-C\lambda_{1}\mathcal{G}\\&-2 {G}^{i\bar{i}}\Big\{[V_{1},\bar{e}_{i}]V_{1}e_{i}(u)+[V_{1},e_{i}]V_{1}\bar{e}_{i}(u)\Big\}. \\
\end{split}
\end{equation}
\textit{Proof.}
By direct calculation,
\[
\begin{split}
    & {G}^{i\bar{i}}(e_{i}\bar{e}_{i}-[e_{i},\bar{e}_{i}]^{0,1})(u_{V_{1}V_{1}})\\
   =& {G}^{i\bar{i}}e_{i}\bar{e}_{i}(V_{1}V_{1}(u)-(\nabla_{V_{1}}V_{1})u)
   -{G}^{i\bar{i}}[e_{i},\bar{e}_{i}]^{0,1}(V_{1}V_{1}(u)-(\nabla_{V_{1}}V_{1})u)\\
   \geq& {G}^{i\bar{i}}V_{1}V_{1}(e_{i}\bar{e}_{i}(u)-[e_{i},\bar{e}_{i}]^{0,1}(u))
   -2{G}^{i\bar{i}}\Big\{[V_{1},\bar{e}_{i}]V_{1}e_{i}(u)+[V_{1},e_{i}]V_{1}\bar{e}_{i}(u)\Big\}\\
   &-{G}^{i\bar{i}}(\nabla_{V_{1}}V_{1})e_{i}\bar{e}_{i}(u)
   +{G}^{i\bar{i}}(\nabla_{V_{1}}V_{1})[e_{i},\bar{e}_{i}]^{0,1}(u)
   -C\lambda_{1}\mathcal{G}\\
   \geq&{G}^{i\bar{i}}V_{1}V_{1}(\tilde{g}_{i\bar{i}})
   -2{G}^{i\bar{i}}\Big\{[V_{1},\bar{e}_{i}]V_{1}e_{i}(u)+[V_{1},e_{i}]V_{1}\bar{e}_{i}(u)\Big\}
   +(\nabla_{V_{1}}V_{1})(\hat{h})-C\lambda_{1}\mathcal{G}.\\
\end{split}
\]
Then the Claim 1 follows if $\lambda_{1}\gg 1$.\qed
~\\

Combining the equalities (\ref{5..13}) and (\ref{claim2}) together, it follows that
\begin{equation}
\begin{split}
\mathcal{L}(\lambda_{1})\geq & 2\sum_{\alpha>1}{G}^{i\bar{i}}\frac{|e_{i}(u_{V_{\alpha}V_{1}})|^{2}}{\lambda_{1}-\lambda_{\alpha}}
-{G}^{i\bar{k},j\bar{l}}V_{1}(\tilde{g}_{i\bar{k}})V_{1}(\tilde{g}_{j\bar{l}})
\\&-2{G}^{i\bar{i}}\textrm{Re}\Big\{[V_{1},e_{i}]V\bar{e}_{i}(u)
+[V_{1},\bar{e}_{i}]Ve_{i}(u)\Big\}-C\lambda_{1}\mathcal{G}.
\end{split}
\end{equation}
Since by (\ref{3.7}) it gives us
\begin{equation}
\mathcal{L}(|\partial u|^{2})\geq  \frac{1}{2}\sum_{j} {G}^{i\bar{i}}(|e_{i}e_{j}u|^{2}+|e_{i}\bar{e}_{j}u|^{2})- C\mathcal{G}.
\end{equation}
Hence we have
\begin{equation}\label{4.10-1}
\begin{split}
\mathcal{L}(Q)
\geq & 2\sum_{\alpha>1}{G}^{i\bar{i}}\frac{|e_{i}(u_{V_{\alpha}V_{1}})|^{2}}{\lambda_{1}(\lambda_{1}-\lambda_{\alpha})}
-\frac{1}{\lambda_{1}} {G}^{i\bar{k},j\bar{l}}V_{1}(\tilde{g}_{i\bar{k}})V_{1}(\tilde{g}_{j\bar{l}})+B^{2}e^{B\widetilde{\eta}} {G}^{i\bar{i}}|e_{i}(\widetilde{\eta})|^{2}
\\&+Be^{B\widetilde{\eta}}\mathcal{L}(\widetilde{\eta})-2 {G}^{i\bar{i}}\frac{\textrm{Re}\{[V_{1},e_{i}]V_{1}\bar{e}_{i}(u)+[V_{1},\bar{e}_{i}]V_{1}e_{i}(u)\}}{\lambda_{1}}
-C\mathcal{G}\\
&-{G}^{i\bar{i}}\frac{|e_{i}(\lambda_{1})|^{2}}{\lambda_{1}^{2}}
+ \frac{\phi'}{2}\sum_{j} {G}^{i\bar{i}}(|e_{i}e_{j}u|^{2}+|e_{i}\bar{e}_{j}u|^{2})+\phi'' {G}^{i\bar{i}}|e_{i}(|\partial u|^{2})|^{2}.\\
\end{split}
\end{equation}
Now we deal with the third derivatives of the right hand.
~\\
\textbf{Claim 2.} For any $\ve\in (0,\frac{1}{2}]$, we have
\begin{equation}\label{4.10}
\begin{split}
   & 2 {G}^{i\bar{i}}\frac{\textrm{Re}\{[V_{1},e_{i}]V_{1}\bar{e}_{i}(u)+[V_{1},\bar{e}_{i}]V_{1}e_{i}(u)\}}{\lambda_{1}}\\
    \leq & \ve{G}^{i\bar{i}}\frac{|e_{i}(\lambda_{1})|^{2}}{\lambda_{1}^{2}}+
    \ve\sum_{\alpha>1}{G}^{i\bar{i}}\frac{|e_{i}(u_{V_{\alpha}V_{1}})|^{2}}{\lambda_{1}(\lambda_{1}-\lambda_{\alpha})}
    +\frac{C}{\ve}\mathcal{G}.
\end{split}
\end{equation}
\begin{proof}
We may find $\mu_{i\beta}\in\mathbb{C}$  such that
\[
[V_{1},e_{i}]=\sum_{\beta=1}^{2n} \mu_{i\beta}V_{\beta},~[V_{1},\bar{e}_{i}]=\sum_{\beta=1}^{2n} \overline{\mu_{i\beta}}V_{\beta}.
\]
Therefore,
\[
\textrm{Re}\Big\{[V_{1},e_{i}]V_{1}\bar{e}_{i}(u)+[V_{1},\bar{e}_{i}]V_{1}e_{i}(u)\Big\}\leq C\sum_{\beta=1}^{2n}|V_{\beta}V_{1}e_{i}(u)|.
\]
Then it suffices to estimate $\underset{\beta}{\sum} {G}^{i\bar{i}}\frac{|V_{\beta}V_{1}e_{i}(u)|}{\lambda_{1}}$.
Since
\[
\begin{split}
  \big|V_{\beta}V_{1}e_{i}(u)\big|= &\big|e_{i}V_{\beta}V_{1}(u)+V_{\beta}[V_{1},e_{i}](u)+[V_{\alpha},e_{i}]V_{1}(u)\big|  \\
    =& \big|e_{i}(u_{V_{\beta}V_{1}})+e_{i}(\nabla_{V_{\beta}}V_{1})(u)+V_{\beta}[V_{1},e_{i}](u)
    +[V_{\alpha},e_{i}]V_{1}(u)\big|\\
    \leq &\big|e_{i}(u_{V_{\beta}V_{1}})\big|+C\lambda_{1},
\end{split}
\]
it follows that
\[
\begin{split}
    \sum_{\beta} {G}^{i\bar{i}}\frac{|V_{\beta}V_{1}e_{i}(u)|}{\lambda_{1}}
    \leq &  \sum_{\beta} {G}^{i\bar{i}}\frac{|e_{i}(u_{V_{\beta}V_{1}})|}{\lambda_{1}}+C\mathcal{G}\\
     \leq & {G}^{i\bar{i}}\frac{|e_{i}(\lambda_{1})|}{\lambda_{1}}+\sum_{\alpha>1} {G}^{i\bar{i}}\frac{|e_{i}(u_{V_{\alpha}V_{1}})|}{\lambda_{1}}+C\mathcal{G}.\\
\end{split}
\]
By the C-S inequality, for $\ve\in (0,\frac{1}{2}]$, we derive
\begin{equation}\label{third order derivative 1}
	\begin{split}
{G}^{i\bar{i}}\frac{|e_{i}(\lambda_{1})|}{\lambda_{1}}\leq \ve {G}^{i\bar{i}}\frac{|e_{i}(\lambda_{1})|^{2}}{\lambda_{1}^{2}}+\frac{C}{\ve}\mathcal{G}
	\end{split}
\end{equation}
and
\begin{equation}\label{the second term}
	\begin{split}
		\sum_{\beta>1} G^{i\bar{i}}\frac{|e_{i}(u_{V_{\beta}V_{1}})|}{\lambda_{1}}\leq & \ve \sum_{\beta>1} G^{i\bar{i}}\frac{|e_{i}(u_{V_{\beta}V_{1}})|^{2}}{\lambda_{1}(\lambda_{1}-\lambda_{\beta})}
		+\sum_{\beta>1}\frac{\lambda_{1}-\lambda_{\beta}}{\ve\lambda_{1}}\mathcal{G} \\
		\leq & \ve\sum_{\beta>1} G^{i\bar{i}}\frac{|e_{i}(u_{V_{\beta}V_{1}})|^{2}}{\lambda_{1}(\lambda_{1}-\lambda_{\beta})}
		+\frac{C}{\ve}\mathcal{G},
	\end{split}
\end{equation}
where we used
$$\sum_{\beta=1}^{2n}\lambda_{\beta}=\Delta u=\Delta^{\mathbb{C}}u+\tau(du)\geq -C+\tau(du)\geq -C$$ (see \cite[Eq. (2.5)]{CTW}) for  the last inequality.
Here  $\tau$ is the torsion vector field of $(\omega, J)$ (the dual of its Lee form, see \cite[Lemma 3.2]{TV07}).
By \eqref{third order derivative 1}-\eqref{the second term},  we have
\[
\sum_{\beta=1}^{2n} G^{i\bar{i}}\frac{|V_{\beta}V_{1}e_{i}(u)|}{\lambda_{1}}\leq
\ve G^{i\bar{i}}\frac{|e_{i}(\lambda_{1})|^{2}}{\lambda_{1}^{2}}+
\ve\sum_{\beta>1} G^{i\bar{i}}\frac{|e_{i}(u_{V_{\beta}V_{1}})|^{2}}{\lambda_{1}(\lambda_{1}-\lambda_{\beta})}
+\frac{C}{\ve}\mathcal{G}.
\]
Then these prove (\ref{4.10}).
\end{proof}
Consequently, Lemma \ref{Lma4.2} follows from (\ref{4.10-1}) and (\ref{4.10}). Now we continue to prove Theorem \ref{Thm4.1}.
\subsection{Continued proof of Theorem \ref{Thm4.1}} The proof can be divided into two cases.
~\\
\textbf{Case 1:} Either
~\\
\textbf{Subcase 1.1:} \begin{equation}\label{4.12}
{G}^{n\bar{n}}\leq B^{3}e^{2B\widetilde{\eta}(0)}{G}^{1\bar{1}},~\mathrm{or}
\end{equation}
\textbf{Subcase 1.2:} \begin{equation}\label{5.15}
\frac{\phi'}{4}\sum_{j} {G}^{i\bar{i}}(|e_{i}e_{j}u|^{2}+|e_{i}\bar{e}_{j}u|^{2})>6\sup_{M}(|\nabla \widetilde{\eta}|^{2})B^{2}e^{2B\widetilde{\eta}}\mathcal{G}.\end{equation}
In this case we can choose $\ve=\frac{1}{2}$.  Using the elemental inequality $|a+b|^{2}\leq 4|a|^{2}+\frac{4}{3}|b|^{2}$ for (\ref{5..10}), we get
\begin{equation}
-(1+\ve){G}^{i\bar{i}}\frac{|e_{i}(\lambda_{1})|^{2}}{\lambda_{1}^{2}}\geq
-6\sup_{M}(|\nabla \widetilde{\eta}|^{2})B^{2}e^{2B\widetilde{\eta}}\mathcal{G}-2(\phi')^{2} {G}^{i\bar{i}}|e_{i}(|\partial u|^{2})|^{2}.
\end{equation}
Plugging it into (\ref{4.7"}) gives us
\begin{equation}
\begin{split}
\mathcal{L}(Q)  \geq & (2-\ve)\sum_{\alpha>1}{G}^{i\bar{i}}\frac{|e_{i}(u_{V_{\alpha}V_{1}})|^{2}}{\lambda_{1}(\lambda_{1}-\lambda_{\alpha})}
-\frac{1}{\lambda_{1}} {G}^{i\bar{k},j\bar{l}}V_{1}(\tilde{g}_{i\bar{k}})V_{1}(\tilde{g}_{j\bar{l}})
\\&-\Big(\frac{C}{\ve}+6\sup_{M}\{|\nabla \widetilde{\eta}|^{2}\}B^{2}e^{2B\widetilde{\eta}}\Big)\mathcal{G}
+ \frac{\phi'}{2}\sum_{j} {G}^{i\bar{i}}(|e_{i}e_{j}u|^{2}+|e_{i}\bar{e}_{j}u|^{2})\\&
+Be^{B\widetilde{\eta}}\mathcal{L}(\widetilde{\eta})+B^{2}e^{B\widetilde{\eta}} {G}^{i\bar{i}}|e_{i}(\widetilde{\eta})|^{2}-C.\\
\end{split}
\end{equation}
\subsubsection{Proof of Subcase 1.1.} In this subcase using the facts of concavity of ${G}$, $\mathcal{L}(\widetilde{\eta})$ has uniform lower bound and $\mathcal{G}\geq \Theta$,  we have
\begin{equation}\label{z}
0\geq \mathcal{L}(Q)\geq \frac{\phi'}{2}\sum_{j} {G}^{i\bar{i}}(|e_{i}e_{j}u|^{2}+|e_{i}\bar{e}_{j}u|^{2})-C_{B}\mathcal{G}.\end{equation}
Where and hereafter $C_{B}$ are positive constants depend on $B$.
Since $\{{G}^{i\bar{i}}\}$ are pairwise comparable (up to a multiplier $B^{3}e^{2B\widetilde{\eta}}$) by (\ref{4.12}), then
\[
\sum_{i,j}(|e_{i}e_{j}u|^{2}+|e_{i}\bar{e}_{j}u|^{2})\leq C_{B}K.
\]
Then the complex covariant derivatives
$u_{ij}=e_{i}e_{j}u-(\nabla_{e_{i}}e_{j})u, ~u_{i\bar{j}}=e_{i}\bar{e}_{j}u-(\nabla_{e_{i}}\bar{e}_{j})u$
satisfy
\[
\sum_{i,j}(|u_{ij}|^{2}+|u_{i\bar{j}}|^{2})\leq C_{B}K,
\]
and this proves (\ref{aa}).
\qed
\subsubsection{Proof of Subcase 1.2.} We also have
\begin{equation}\label{4.19}
\begin{split}
\mathcal{L}(Q)  \geq &-\frac{C}{\ve}\mathcal{G}-6\sup_{M}|\nabla \widetilde{\eta}|^{2}B^{2}e^{2B\widetilde{\eta}}\mathcal{G}-C
\\&+ \frac{\phi'}{2}\sum_{j} {G}^{i\bar{i}}(|e_{i}e_{j}u|^{2}+|e_{i}\bar{e}_{j}u|^{2})
+Be^{B\widetilde{\eta}}\mathcal{L}(\widetilde{\eta})
\\ \geq & \frac{\phi'}{4}\sum_{j} {G}^{i\bar{i}}(|e_{i}e_{j}u|^{2}+|e_{i}\bar{e}_{j}u|^{2})
-\frac{C}{\ve}\mathcal{G}
+Be^{B\widetilde{\eta}}\mathcal{L}(\widetilde{\eta}),
\end{split}
\end{equation}
where we have used (\ref{5.15}) in the last inequality.
~\\
(a). If (\ref{2..25}) holds, then by (\ref{4.19}) and the fact
$\mathcal{L}(\widetilde{\eta})\geq \theta\mathcal{G}\geq \frac{1}{2}\theta(\mathcal{G}+\Theta)$
we have
\[
\mathcal{L}(Q)  \geq \frac{\phi'}{4}\sum_{j} {G}^{i\bar{i}}(|e_{i}e_{j}u|^{2}+|e_{i}\bar{e}_{j}u|^{2})+\Big(\frac{1}{2}\theta Be^{B\widetilde{\eta}}-\frac{C}{\ve}\Big)\mathcal{G}+\frac{1}{2}\theta Be^{B\widetilde{\eta}}\Theta.
\]
This yields a contradiction if we further assume $B$ is large enough.
~\\
(b). If (\ref{2..26}) holds, then by (\ref{4.19}) and (\ref{4.31}) we obtain
\[
\mathcal{L}(Q)\geq\frac{\phi'}{4}\sum_{j} {G}^{i\bar{i}}(|e_{i}e_{j}u|^{2}+|e_{i}\bar{e}_{j}u|^{2})-C_{B}
\]
if $B$ is large enough such that $\frac{C}{\ve}\mathcal{G}$ in (\ref{4.19}) can be discarded. Moreover, using the fact
\[{G}^{k\bar{k}}\geq \theta\mathcal{G}\geq \theta\Theta,~\forall k=1,2,\cdots,n,\]
we also have
\[\sum_{j}(|e_{i}e_{j}u|^{2}+|e_{i}\bar{e}_{j}u|^{2})\leq C_{B}K.\]
The rest of the proof is same as Subcase 1.1.
~\\
\textbf{Case 2:} If the Case 1 does not hold, define the index set
\[
I=\big\{1\leq i\leq n:~ {G}^{n\bar{n}}\geq B^{3}e^{2B\widetilde{\eta}}{G}^{i\bar{i}}\big\}.
\]~\\
Observe that $1\in I$, $n\notin I$. Hence we may assume $I=\{1,2,\cdots, p\}$ for $p<n$.
\begin{lemma} \cite[Lemma 5.5]{CTW} Assume $B\geq 6n\sup_{M}|\nabla \widetilde{\eta}|^{2}$. At the origin, we have
\begin{equation}
-(1+\ve)\sum_{i\in I}{G}^{i\bar{i}} \frac{|e_{i}(\lambda_{1})|^{2}}{\lambda_{1}^{2}}\geq -\mathcal{G}-2(\phi')^{2}\sum_{i\in I}{G}^{i\bar{i}}|e_{i}(|\partial u|^{2})|.
\end{equation}
\end{lemma}
Define a new (1,0) vector field by
\[
\widetilde{e}_{1}=\frac{1}{\sqrt{2}}(V_{1}-\sqrt{-1}JV_{1}).
\]
At the origin, we can find a sequence of complex numbers $\nu_{1},\cdots, \nu_{n}$ such that
\[
\widetilde{e}_{1}=\sum_{1}^{n}\nu_{k}e_{k},~\sum_{1}^{n}|\nu_{k}|^{2}=1.
\]
\begin{lemma}\cite[Lemma 5.6]{CTW}\label{coeff} We have
\[
|\nu_{k}|\leq \frac{C_{B}}{\lambda_{1}} ~\textrm{for all} ~k\notin I.
\]
\end{lemma}
Once we proved it, now we can estimate the first three terms in Lemma \ref{Lma4.2}. Since $JV_{1}$ is $g$-unit and $g$-orthogonal to $V_{1}$, then we can find real numbers $\mu_{2},\cdots,\mu_{2n}$ such that
\[
JV_{1}=\sum_{\alpha>1}\mu_{\alpha}V_{\alpha}, ~\sum_{\alpha>1}\mu_{\alpha}^{2}=1~ \textrm{at the origin}.
\]
The following lemma is key to our estimate.
\begin{lemma}\label{key inequality} For any constant $\gamma>0$, we have
\[
\begin{split}
    & (2-\ve)\sum_{\alpha>1}{G}^{i\bar{i}}\frac{|e_{i}(u_{V_{\alpha}V_{1}})|^{2}}{\lambda_{1}(\lambda_{1}-\lambda_{\alpha})}
-\frac{1}{\lambda_{1}} {G}^{i\bar{k},j\bar{l}}V_{1}(\tilde{g}_{i\bar{k}})V_{1}(\tilde{g}_{j\bar{l}})
     -(1+\ve)\sum_{i\notin I}{G}^{i\bar{i}} \frac{|e_{i}(\lambda_{1})|^{2}}{\lambda_{1}^{2}}\\
   \geq &(2-\ve
    )\sum_{i\notin I}\sum_{\alpha>1}{G}^{i\bar{i}} \frac{|e_{i}(u_{V_{\alpha}V_{1}})|^{2}}{\lambda_{1}(\lambda_{1}-\lambda_{\alpha})}
    +\sum_{k\in I,i\notin I}\frac{2}{\lambda_{1}} {G}^{i\bar{i}}\tilde{g}^{k\bar{k}}|V_{1}(\tilde{g}_{i\bar{k}})|^{2}\\
    &-3\ve\sum_{i\notin I}{G}^{i\bar{i}}\frac{|e_{i}(\lambda_{1})|^{2}}{\lambda_{1}^{2}}-
    2(1-\ve)(1+{\gamma})\tilde{g}_{\tilde{1}\bar{\tilde{1}}}\sum_{k\in I,i\notin I}{G}^{i\bar{i}}\tilde{g}^{k\bar{k}}\frac{|V_{1}(\tilde{g}_{i\bar{k}})|^{2}}{\lambda_{1}^{2}}\\
    &-\frac{C}{\ve}\mathcal{G}-(1-\ve)(1+\frac{1}{\gamma})(\lambda_{1}-\sum_{\alpha>1}\lambda_{\alpha}
    \mu_{\alpha}^{2})
    \sum_{i\notin I}\sum_{\alpha>1}\frac{{G}^{i\bar{i}}}{\lambda_{1}^{2}}
    \frac{|e_{i}(u_{V_{\alpha}V_{1}})|^{2}}{\lambda_{1}-\lambda_{\alpha}}\\
\end{split}
\]
if we assume $\lambda_{1}\geq \frac{n^{2}C_{B}}{\ve}$, where $\tilde{g}_{\tilde{1}\bar{\tilde{1}}}=\sum \tilde{g}_{i\bar{i}}|\nu_{i}|^{2}$.
\end{lemma}
~\\
\textit{Proof.} \textrm{Step 1:} We can prove
\begin{equation}\label{4.26}
e_{i}(\lambda_{1})=\sqrt{2}\sum_{k} \overline{\nu_{k}}V_{1}(\tilde{g}_{i\bar{k}})-\sqrt{-1}\sum_{\alpha>1}\mu_{\alpha}e_{i}(u_{V_{1}
V_{\alpha}})+O(\lambda_{1}),\end{equation}
where $O(\lambda_{1})$ denotes the terms those can be controlled by $\lambda_{1}$. Indeed, since $\overline{\widetilde{e}_{1}}=\frac{1}{\sqrt{2}}(V_{1}+\sqrt{-1}JV_{1})$, therefore,
\[
e_{i}(\lambda_{1})=\sqrt{2} e_{i}(u_{V_{1}\overline{\widetilde{e}_{1}}})-\sqrt{-1} e_{i}(u_{V_{1}JV_{1}}).
\]
The first term
\begin{equation}\label{4.27}
\begin{split}
e_{i}(u_{V_{1}\overline{\widetilde{e}_{1}}})= & e_{i}(V_{1}\overline{\widetilde{e}_{1}}u-(\nabla_{V_{1}}\overline{\widetilde{e}_{1}})u)
=\overline{\widetilde{e}_{1}}e_{i}V_{1}u+O(\lambda_{1})   \\
= & \sum_{k}\overline{\nu_{k}}V_{1}(\tilde{g}_{i\bar{k}})+O(\lambda_{1}).
\end{split}
\end{equation}
The second term
\begin{equation}\label{4.28}
\begin{split}
e_{i}(u_{V_{1}JV_{1}})=&e_{i}{V_{1}JV_{1}}(u)+O(\lambda_{1})=JV_{1}e_{i}{V_{1}}(u)+O(\lambda_{1})  \\
=&\sum_{\alpha>1}V_{\alpha}e_{i}{V_{1}}(u)+O(\lambda_{1})=\sum_{\alpha>1}e_{i}(u_{V_{\alpha}{V_{1}}})+O(\lambda_{1}).\\
\end{split}
\end{equation}
Thus, (\ref{4.26}) follows from (\ref{4.27}) and (\ref{4.28}).
~\\
Step 2: Hence we have
\[
\begin{split}
&-(1+\ve)\sum_{i\notin I} {G}^{i\bar{i}}\frac{|e_{i}(\lambda_{1})|^{2}}{\lambda_{1}^{2}}\\
\geq &
    -(1-\ve)\sum_{i\notin I} {G}^{i\bar{i}}\frac{|\sqrt{2}\sum_{k\in I} \overline{\nu_{k}}V_{1}(\tilde{g}_{i\bar{k}})-\sqrt{-1}\sum_{\alpha>1}\mu_{\alpha}e_{i}(u_{V_{1}
V_{\alpha}})|^{2}}{\lambda_{1}^{2}}\\
&-3\ve \sum_{i\notin I} {G}^{i\bar{i}}\frac{|e_{i}(\lambda_{1})|^{2}}{\lambda_{1}^{2}}-\frac{C_{B}}{\ve}\sum_{i\notin I,k\notin I}{G}^{i\bar{i}}\frac{|V_{1}(\tilde{g}_{i\bar{k}})|^{2}}{\lambda_{1}^{4}}-\frac{C}{\ve}\mathcal{G},
\end{split}
\]
where we used the Lemma \ref{coeff}.
~\\
Step 3: Using the C-S inequality, we have
\begin{equation*}
	\begin{split}
		&\Big|\sum_{\alpha>1}\mu_{\alpha}e_{i}(u_{V_{1}V_{\alpha}})\Big|^{2}\leq \sum_{\alpha>1}(\lambda_{1}-\lambda_{\alpha}\mu_{\alpha}^{2})
		\sum_{\beta>1}\frac{|e_{i}(u_{V_{1}V_{\beta}})|^{2}}{\lambda_{1}-\lambda_{\beta}},\\
		&\Big|\sum_{k\in I}\overline{\nu_{k}}V_{1}(\tilde{g}_{i\bar{k}})\Big|^{2}\leq \Big(\sum_{i}\tilde{g}_{i\bar{i}}|\nu_{i}|^{2}\Big)\sum_{k\in I}\tilde{g}^{k\bar{k}}|V_{1}(\tilde{g}_{i\bar{k}})|^{2}.
	\end{split}
\end{equation*}
Then for each $\gamma>0$, using the C-S inequality again to get
\[
\begin{split}
   & (1-\ve)\sum_{i\notin I} {G}^{i\bar{i}}\frac{|\sqrt{2}\sum_{k\in I} \overline{\nu_{k}}V_{1}(\tilde{g}_{i\bar{k}})-\sqrt{-1}\sum_{\alpha>1}\mu_{\alpha}e_{i}(u_{V_{1}
V_{\alpha}})|^{2}}{\lambda_{1}^{2}}\\
\leq & 2(1-\ve)(1+\gamma)\sum_{i\notin I} {G}^{i\bar{i}}\frac{|\sum_{k\in I} \overline{\nu_{k}}V_{1}(\tilde{g}_{i\bar{k}})|^{2}}{\lambda_{1}^{2}}\\
&+(1-\ve)(1+\frac{1}{\gamma})\sum_{i\notin I} {G}^{i\bar{i}}\frac{|\sum_{\alpha>1}\mu_{\alpha}e_{i}(u_{V_{1}
V_{\alpha}})|^{2}}{\lambda_{1}^{2}}\\
\leq & 2(1-\ve)(1+\gamma)\tilde{g}_{\tilde{1}\bar{\tilde{1}}}\sum_{i\notin I}\sum_{k\in I} \frac{{G}^{i\bar{i}}}{\lambda_{1}^{2}} \tilde{g}^{k\bar{k}}|V_{1}(\tilde{g}_{i\bar{k}})|^{2}\\
&+(1-\ve)(1+\frac{1}{\gamma})(\lambda_{1}-\sum_{\alpha>1} \lambda_{\alpha}\mu_{\alpha}^{2})\sum_{i\notin I}\sum_{\alpha>1}\frac{{G}^{i\bar{i}}}{\lambda_{1}^{2}} \frac{|e_{i}(u_{V_{\alpha}V_{1}})|^{2}}{\lambda_{1}-\lambda_{\alpha}}.
\end{split}
\]
Step 4: By direct calculation,
\begin{equation}\label{5..25}
\begin{split}
&-{G}^{i\bar{k},j\bar{l}}V_{1}(\tilde{g}_{i\bar{k}})V_{1}(\tilde{g}_{j\bar{l}})\\
 = & \sum_{i\neq k}\sigma_{m-2;ik}(\tilde{g}^{i\bar{i}})^{2}(\tilde{g}^{k\bar{k}})^{2}
    \big\{V_{1}(\tilde{g}_{i\bar{i}})V_{1}
    (\tilde{g}_{k\bar{k}})-|V_{1}(\tilde{g}_{i\bar{k}})|^{2}\big\}
    +2\sum_{i,k}{G}^{i\bar{i}}\tilde{g}^{k\bar{k}}|V_{1}(\tilde{g}_{i\bar{k}})|^{2}\\
 =&2\sum_{i\neq k}\sigma_{m-1;i}(\tilde{g}^{i\bar{i}})^{2}\tilde{g}^{k\bar{k}}|V_{1}(\tilde{g}_{i\bar{k}})|^{2}
    +2\sum_{i=1}^{n}\sigma_{m-1;i}(\tilde{g}^{i\bar{i}})^{3}|V_{1}(\tilde{g}_{i\bar{i}})|^{2}\\
&+\sum_{i\neq k}\sigma_{m-2;ik}(\tilde{g}^{i\bar{i}})^{2}(\tilde{g}^{k\bar{k}})^{2}
    \big\{V_{1}(\tilde{g}_{i\bar{i}})V_{1}(\tilde{g}_{k\bar{k}})-|V_{1}(\tilde{g}_{i\bar{k}})|^{2}\big\}\\
\geq&\sum_{i\neq k}\sigma_{m-1;i}(\tilde{g}^{i\bar{i}})^{2}\tilde{g}^{k\bar{k}}|V_{1}(\tilde{g}_{i\bar{k}})|^{2}
    +2\sum_{i=1}^{n}\sigma_{m-1;i}(\tilde{g}^{i\bar{i}})^{3}|V_{1}(\tilde{g}_{i\bar{i}})|^{2}
\\&+\sum_{i\neq k}\sigma_{m-2;ik}(\tilde{g}^{i\bar{i}})^{2}(\tilde{g}^{k\bar{k}})^{2}
    V_{1}(\tilde{g}_{i\bar{i}})V_{1}(\tilde{g}_{k\bar{k}}),\\
\end{split}
\end{equation}
where we used the following inequality (see \cite{GS13}):
\[
\sum_{i\neq k}(\sigma_{m-1;i}-\sigma_{m-2;ik}\tilde{g}^{k\bar{k}})(\tilde{g}^{i\bar{i}})^{2}\tilde{g}^{k\bar{k}}
|V_{1}(\tilde{g}_{i\bar{k}})|^{2}\geq 0.
\]
We also need the next inequality from \cite{GLZ09}, see also \cite{FLM11,GS13}:
\[
\sum_{i}\frac{\sigma_{m-1;i}(\tau)}{\tau_{i}}\xi_{i}\bar{\xi}_{i}
+\sum_{i\neq k}\sigma_{m-2;ik}(\tau)\xi_{i}\bar{\xi}_{k}\geq \sum_{i,k}\frac{\sigma_{m-1;i}(\tau)\sigma_{m-1;k}(\tau)}{\sigma_{m}(\tau)}\xi_{i}\bar{\xi}_{k}
= 0
\]
for every $\tau=(\tau_{1},\cdots,\tau_{n})\in \Gamma_{n}$ and $(\xi_{1},\cdots,\xi_{n})\in \mathbb{C}^{n}$. Choose $\tau=(\tilde{g}^{1\bar{1}},\cdots,\tilde{g}^{n\bar{n}})$ and $\xi_{i}=V_{1}(\tilde{g}^{i\bar{i}})$, then
\begin{equation}\label{5.28-1}
\sum_{i=1}^{n}\sigma_{m-1;i}(\tilde{g}^{i\bar{i}})^{3}|V_{1}(\tilde{g}_{i\bar{i}})|^{2}+\sum_{i\neq k}\sigma_{m-2;ik}(\tilde{g}^{i\bar{i}})^{2}(\tilde{g}^{k\bar{k}})^{2}
    V_{1}(\tilde{g}_{i\bar{i}})V_{1}(\tilde{g}_{k\bar{k}})\geq 0.
\end{equation}
It follows (\ref{5..25})-(\ref{5.28-1}) that
\begin{equation}\label{5..27}
\begin{split}
&-{G}^{i\bar{k},j\bar{l}}V_{1}(\tilde{g}_{i\bar{k}})V_{1}(\tilde{g}_{j\bar{l}})
\geq\sum_{i\neq
k}{G}^{i\bar{i}}\tilde{g}^{k\bar{k}}|V_{1}(\tilde{g}_{i\bar{k}})|^{2}
    +\sum_{i=1}^{n}{G}^{i\bar{i}}\tilde{g}^{i\bar{i}}|V_{1}(\tilde{g}_{i\bar{i}})|^{2}.\\
\end{split}
\end{equation}
If $\tilde{g}_{i\bar{i}}\geq \tilde{g}_{k\bar{k}}$, we have $\sigma_{m-1;i}\tilde{g}^{i\bar{i}}\leq \sigma_{m-1;k}\tilde{g}^{k\bar{k}}$. Hence
\begin{equation}\label{5.27-1}
\sum_{i\in I,k\notin I}{G}^{i\bar{i}}\tilde{g}^{k\bar{k}}|V_{1}(\tilde{g}_{i\bar{k}})|^{2}\geq \sum_{i\notin I,k\in I}{G}^{i\bar{i}}\tilde{g}^{k\bar{k}}|V_{1}(\tilde{g}_{i\bar{k}})|^{2}.
\end{equation}
On the one hand, since $\{(i,k):1\leq i,k\leq n,i\in I,k\notin I\}$, $\{(i,k):1\leq i,k\leq n,i\notin I,k\in I\}$ and  $\{(i,k):1\leq i\neq k\leq n,i,k\notin I\}$ are pairwise disjoint subsets of $\{(i,k):1\leq i\neq k\leq n\}$, and by (\ref{5.27-1}), to show
\begin{equation}\label{5.27}
\begin{split}
\sum_{i\neq k}{G}^{i\bar{i}}\tilde{g}^{k\bar{k}}|V_{1}(\tilde{g}_{i\bar{k}})|^{2}
\geq& 2\sum_{i\notin I,k\in I}{G}^{i\bar{i}}\tilde{g}^{k\bar{k}}|V_{1}(\tilde{g}_{i\bar{k}})|^{2}
\\&+\frac{C_{B}}{\ve\lambda_{1}^{3}}\sum_{i\neq k,i\notin I,k\notin I}{G}^{i\bar{i}}{|V_{1}(\tilde{g}_{i\bar{k}})|^{2}},
\end{split}
\end{equation}
we shall prove
 \begin{equation}\label{inequ}
   \tilde{g}^{k\bar{k}}\geq \frac{C_{B}}{\ve\lambda_{1}^{3}}
 \end{equation}
for each $k$. Since $\tilde{g}_{1\bar{1}}$ is comparable to $\lambda_{1}$ and $\lambda_{1}\geq \frac{n^{2}C_{B}}{\ve}$, then $\tilde{g}^{1\bar{1}}\geq \frac{C_{B}}{\ve\lambda_{1}^{3}}$. This proves \eqref{inequ} because $\tilde{g}^{k\bar{k}}\geq \tilde{g}^{1\bar{1}}$. On the other hand, it  also implies that
\begin{equation}\label{5.28}
\sum_{i=1}^{n}{G}^{i\bar{i}}\tilde{g}^{i\bar{i}}|V_{1}(\tilde{g}_{i\bar{i}})|^{2}\geq \frac{C_{B}}{\ve\lambda_{1}^{3}}\sum_{i\notin I}{G}^{i\bar{i}}{|V_{1}(\tilde{g}_{i\bar{i}})|^{2}}.
\end{equation}It follows from (\ref{5..27}), (\ref{5.27}) and (\ref{5.28}) that
\begin{equation}
\begin{split}
&-\frac{1}{\lambda_{1}} {G}^{i\bar{k},j\bar{l}}V_{1}(\tilde{g}_{i\bar{k}})V_{1}(\tilde{g}_{j\bar{l}})\\
\geq &
\frac{2}{\lambda_{1}}\sum_{k\in I,i\notin I}{G}^{i\bar{i}}\tilde{g}^{k\bar{k}}|V_{1}(\tilde{g}_{i\bar{k}})|^{2}
+\frac{C_{B}}{\ve\lambda_{1}^{4}}\sum_{i\notin I,k\notin I}{G}^{i\bar{i}}{|V_{1}(\tilde{g}_{i\bar{k}})|^{2}}.\\
\end{split}
\end{equation}
~\\
Consequently, the lemma follows from the previous several steps.\qed
\begin{lemma}\label{Lma5.6} If we assume $\lambda_{1}\geq C/{\ve^{3}}$. Then
\[
\begin{split}
&(2-\ve)\sum_{\alpha>1}{G}^{i\bar{i}}\frac{|e_{i}(u_{V_{\alpha}V_{1}})|^{2}}
{\lambda_{1}(\lambda_{1}-\lambda_{\alpha})}-\frac{1}{\lambda_{1}} {G}^{i\bar{k},j\bar{l}}V_{1}(\tilde{g}_{i\bar{k}})V_{1}(\tilde{g}_{j\bar{l}})
-(1+\ve)\sum_{i\notin I} {G}^{i\bar{i}}\frac{|e_{i}(\lambda_{1})|^{2}}{\lambda_{1}^{2}}\\
\geq& -6\ve B^{2}e^{2B\widetilde{\eta}}\sum_{i} {G}^{i\bar{i}}|e_{i}(\widetilde{\eta})|^{2}-6\ve (\phi')^{2}\sum_{i\notin I}{G}^{i\bar{i}}|e_{i}(|\partial u|^{2})|^{2}-\frac{C}{\ve}\mathcal{G}.
\end{split}
\]
\end{lemma}
\textit{Proof.} It suffices to prove
\begin{equation}\label{4.30}
\begin{split}
(2-\ve)&\sum_{\alpha>1}{G}^{i\bar{i}}\frac{|e_{i}(u_{V_{\alpha}V_{1}})|^{2}}
{\lambda_{1}(\lambda_{1}-\lambda_{\alpha})}-\frac{1}{\lambda_{1}} {G}^{i\bar{k},j\bar{l}}V_{1}(\tilde{g}_{i\bar{k}})V_{1}(\tilde{g}_{j\bar{l}})\\&
-(1+\ve)\sum_{i\notin I} {G}^{i\bar{i}}\frac{|e_{i}(\lambda_{1})|^{2}}{\lambda_{1}^{2}}
\geq -3\ve \sum_{i\notin I} {G}^{i\bar{i}}\frac{|e_{i}(\lambda_{1})|^{2}}{\lambda_{1}^{2}}-\frac{C}{\ve} \mathcal{G}.
\end{split}
\end{equation}
We divide the proof into two conditions.
~\\
\textbf{Condition 1:} Assume that
\begin{equation}\label{5.32}
\lambda_{1}+\sum_{\alpha>1}\lambda_{\alpha}\mu_{\alpha}^{2}\geq 2(1-\ve)\tilde{g}_{\tilde{1}\bar{\tilde{1}}}>0.
\end{equation}
\textit{Proof of Condition 1} It follows from Lemma \ref{key inequality} and (\ref{5.32}) that
\begin{equation}\label{5.76}
\begin{split}
    & (2-\ve)\sum_{\alpha>1}{G}^{i\bar{i}}\frac{|e_{i}(u_{V_{\alpha}V_{1}})|^{2}}{\lambda_{1}(\lambda_{1}-\lambda_{\alpha})}
-\frac{1}{\lambda_{1}} {G}^{i\bar{k},j\bar{l}}V_{1}(\tilde{g}_{i\bar{k}})V_{1}(\tilde{g}_{j\bar{l}})\\
    & -(1+\ve)\sum_{i\notin I} {G}^{i\bar{i}}\frac{|e_{i}(\lambda_{1})|^{2}}{\lambda_{1}^{2}}\\
    \geq &\sum_{i\notin I}\sum_{\alpha>1} \frac{{G}^{i\bar{i}}}{\lambda_{1}^{2}}\left(\frac{(2-\ve
    )\lambda_{1}}{\lambda_{1}-\lambda_{\alpha}}|e_{i}(u_{V_{\alpha}V_{1}})|^{2}\right)
    +\sum_{k\in I,i\notin I}\frac{2}{\lambda_{1}} {G}^{i\bar{i}}\tilde{g}^{k\bar{k}}|V_{1}(\tilde{g}_{i\bar{k}})|^{2}\\
    &-3\ve\sum_{i\notin I}{G}^{i\bar{i}}\frac{|e_{i}(\lambda_{1})|^{2}}{\lambda_{1}^{2}}-
    (1+{\gamma})(\lambda_{1}+\sum_{\alpha>1}\lambda_{\alpha}\mu_{\alpha}^{2})\sum_{k\in I,i\notin I}{G}^{i\bar{i}}\tilde{g}^{k\bar{k}}\frac{|V_{1}(\tilde{g}_{i\bar{k}})|^{2}}{\lambda_{1}^{2}}\\
    &-\frac{C}{\ve}\mathcal{G}-
    (1-\ve)(1+\frac{1}{\gamma})(\lambda_{1}-\sum_{\alpha>1}\lambda_{\alpha}\mu_{\alpha}^{2})
    \sum_{i\notin I}\sum_{\alpha>1}\frac{{G}^{i\bar{i}}}{\lambda_{1}^{2}}\frac{|e_{i}(u_{V_{\alpha}V_{1}})|^{2}}{\lambda_{1}-\lambda_{\alpha}}.
\end{split}
\end{equation}
We only need to choose
\[
\gamma=\frac{\lambda_{1}-\underset{\alpha>1}{\sum}\lambda_{\alpha}\mu_{\alpha}^{2}}
{\lambda_{1}+\underset{\alpha>1}{\sum}\lambda_{\alpha}\mu_{\alpha}^{2}}.
\]~\\
Then on the right hand of (\ref{5.76}), the first term cancels the last term and the second term cancels the fourth term. This proves  (\ref{4.30}).
\qed
~\\
\textbf{Condition 2:}  Assume that
\begin{equation}\label{4.31-1}
{\lambda_{1}+\sum_{\alpha>1}\lambda_{\alpha}\mu_{\alpha}^{2}}< 2(1-\ve)\tilde{g}_{\tilde{1}\bar{\tilde{1}}}.
\end{equation}
~\\
\textit{Proof of Condition 2} By a similar calculation in \cite{CTW}, we have
\begin{equation}\label{4.32}
0<\tilde{g}_{\tilde{1}\bar{\tilde{1}}}\leq \frac{1}{2}(\lambda_{1}+\sum_{\alpha>1}\lambda_{\alpha}\mu_{\alpha}^{2})+C.
\end{equation}
Plugging it into (\ref{4.31-1}), then
$\lambda_{1}+\sum_{\alpha>1} \lambda_{\alpha}\mu_{\alpha}^{2}\geq -C$
and $\tilde{g}_{\tilde{1}\bar{\tilde{1}}}\leq C/{\ve}$. Hence,
\[
0<\lambda_{1}-\sum_{\alpha>1} \lambda_{\alpha}\mu_{\alpha}^{2}\leq 2\lambda_{1}+C\leq (2+2\ve^{2})\lambda_{1}
\]
provided $\lambda_{1}\geq C/{\ve^{2}}$. Choose $\gamma=1/{\ve^{2}}$, it follows that
\[
\begin{split}
(1-\ve)(1+\frac{1}{\gamma})(\lambda_{1}-\sum_{\alpha>1}\lambda_{\alpha}\mu_{\alpha}^{2})
 \leq & 2(1-\ve)(1+\ve^{2})^{2}\lambda_{1}
 \leq (2-\ve)\lambda_{1}.
 \end{split}
\]
Plugging it into Lemma \ref{key inequality} yields
\[
\begin{split}
    & (2-\ve)\sum_{\alpha>1}{G}^{i\bar{i}}\frac{|e_{i}(u_{V_{\alpha}V_{1}})|^{2}}{\lambda_{1}(\lambda_{1}-\lambda_{\alpha})}
-\frac{1}{\lambda_{1}} {G}^{i\bar{k},j\bar{l}}V_{1}(\tilde{g}_{i\bar{k}})V_{1}(\tilde{g}_{j\bar{l}}) -(1+\ve)\sum_{i\notin I} {G}^{i\bar{i}}\frac{|e_{i}(\lambda_{1})|^{2}}{\lambda_{1}^{2}}\\
    \geq &2\sum_{k\in I,i\notin I} {G}^{i\bar{i}}\tilde{g}^{k\bar{k}}\frac{|V_{1}(\tilde{g}_{i\bar{k}})|^{2}}{\lambda_{1}}-3\ve\sum_{i\not\in I}{G}^{i\bar{i}}\frac{|e_{i}(\lambda_{1})|^{2}}{\lambda_{1}^{2}}\\
    &-2(1-\ve)(1+\frac{1}{\ve^{2}})\tilde{g}_{\tilde{1}\bar{\tilde{1}}}\sum_{k\in I,i\not\in I}{G}^{i\bar{i}}\tilde{g}^{k\bar{k}}\frac{|V_{1}(\tilde{g}_{i\bar{k}})|^{2}}{\lambda_{1}^{2}}    -\frac{C}{\ve}\mathcal{G}\\
    \geq &2\sum_{k\in I,i\not\in I} {G}^{i\bar{i}}\tilde{g}^{k\bar{k}}\frac{|V_{1}(\tilde{g}_{i\bar{k}})|^{2}}{\lambda_{1}} -3\ve\sum_{i\not\in I}{G}^{i\bar{i}}\frac{|e_{i}(\lambda_{1})|^{2}}{\lambda_{1}^{2}}\\
    &-(1-\ve)(1+\frac{1}{\ve^{2}})\frac{C}{\ve}\sum_{k\in I}\sum_{i\not\in I}{G}^{i\bar{i}}\tilde{g}^{k\bar{k}}\frac{|V_{1}(\tilde{g}_{i\bar{k}})|^{2}}{\lambda_{1}^{2}}    -\frac{C}{\ve}\mathcal{G}\\
    \geq &-3\ve\sum_{i\not\in I}{G}^{i\bar{i}}\frac{|e_{i}(\lambda_{1})|^{2}}{\lambda_{1}^{2}}-\frac{C}{\ve}\mathcal{G},\\
\end{split}
\]
~\\
if we assume $\lambda_{1}\geq C/{\ve^{3}}$ in the last inequality. This proves  (\ref{4.30}) and hence the proof of the lemma is completely.\qed
~\\

We now going to complete the proof of second order estimate. Plugging Lemma \ref{Lma5.6} into (\ref{4.7"}) we have
\begin{equation}\label{}
\begin{split}
\mathcal{L}(Q)  \geq & -6\ve B^{2}e^{2B\widetilde{\eta}} {G}^{i\bar{i}}|e_{i}(\widetilde{\eta})|^{2}-6\ve (\phi')^{2}\sum_{i\not\in I}{G}^{i\bar{i}}|e_{i}(|\partial u|^{2})|^{2}-\frac{C}{\ve}\mathcal{G}\\
&+ \frac{\phi'}{2}\sum_{j} {G}^{i\bar{i}}(|e_{i}e_{j}u|^{2}+|e_{i}\bar{e}_{j}u|^{2})
+B^{2}e^{B\widetilde{\eta}} {G}^{i\bar{i}}|e_{i}(\widetilde{\eta})|^{2}+Be^{B\widetilde{\eta}}\mathcal{L}(\widetilde{\eta})\\
& +\phi''{G}^{i\bar{i}}|e_{i}(|\partial u|^{2})|^{2}-2(\phi')^{2}\sum_{i\in I}{G}^{i\bar{i}}|e_{i}(|\partial u|^{2})|^{2}.
\end{split}
\end{equation}
Choosing $\ve< \min\{\frac{1}{6n},\frac{\theta}{6}\}$ such that $e^{B\widetilde{\eta}(0)}=\frac{1}{6\ve}$ (this is possible if $B$ large) and by $\phi''=2(\phi')^{2}$, then
\[
\frac{B}{6\ve}\mathcal{L}(\widetilde{\eta})-\frac{C}{\ve}\mathcal{G}+ \frac{\phi'}{2}\sum_{j} {G}^{i\bar{i}}(|e_{i}e_{j}u|^{2}+|e_{i}\bar{e}_{j}u|^{2})\leq 0.
\]
(a). Suppose (\ref{2..25}) holds. Then we have
\[
(\frac{B\theta}{6\ve}-\frac{C}{\ve})\mathcal{G}+ \frac{\phi'}{2}\sum_{j} {G}^{i\bar{i}}(|e_{i}e_{j}u|^{2}+|e_{i}\bar{e}_{j}u|^{2})\leq 0.
\]
~\\
Choosing $B$ sufficiently large and $\ve<\theta/6$ is small enough such that $B\theta/6-C\geq B\ve$. Then at the origin we have
\[
0\geq B\mathcal{G}+ \frac{\phi'}{2}\sum_{j} {G}^{i\bar{i}}(|e_{i}e_{j}u|^{2}+|e_{i}\bar{e}_{j}u|^{2}).
\]
This yields a contradiction.
~\\
(b). Suppose  (\ref{2..26}) holds. With the aid of (\ref{4.31}) then we have
\[
\frac{B}{6\ve}(\tau\mathcal{G}-C)-\frac{C}{\ve}\mathcal{G}+ \frac{\phi'}{2}\sum_{j} {G}^{i\bar{i}}(|e_{i}e_{j}u|^{2}+|e_{i}\bar{e}_{j}u|^{2})\leq 0.
\]
Hence
\[
\frac{\phi'}{2}\sum_{j} {G}^{i\bar{i}}(|e_{i}e_{j}u|^{2}+|e_{i}\bar{e}_{j}u|^{2})\leq CBe^{B\widetilde{\eta}}
\]
if we assume $ B\tau/6>C$. Using the fact
${G}^{i\bar{i}}\geq \theta\mathcal{G}\geq \theta\Theta$ for  each $i$.
Therefore,
\begin{equation}\label{}
\frac{\phi'\theta\Theta}{2}\sum_{i,j} (|e_{i}e_{j}u|^{2}+|e_{i}\bar{e}_{j}u|^{2})\leq CBe^{B\widetilde{\eta}}.
\end{equation}
The rest of proof can be found in Subcase 1.1, so we omit it here. 
\qed
\section{Proof of Theorem 1.2}
 Our proof is close to the Hermitian setting which was given by Sun \cite{Sun1}. 
For an arbitrary admissible function $v$ satisfies  \eqref{supersolution}. Define
$\tilde{\psi}=\frac{\chi_{v}^{n}}{\chi_{v}^{n-m}\wedge\omega^{m}}\leq \psi$. Consider the flow
\begin{equation}\label{flow}
\chi_{v+u_{t}}^{n}=\psi^{t}\tilde{\psi}^{1-t}e^{b_{t}}\chi_{v+u_{t}}^{n-m}\wedge\omega^{m}, ~\textrm{for}~t\in [0,1],
\end{equation}
with $u_{t}\in \mathcal{H}(M,\chi_{v})$. Here $b_{t}$ is a constant for each $t$ with $b_{0}=0$.

Denote $\psi_{t}=\psi^{t}\tilde{\psi}^{1-t}e^{b_{t}}$. We claim that
\begin{equation}\label{}
\psi_{t}\leq \psi, ~\textrm{for}~t\in [0,1].
\end{equation}
Indeed, at the maximum point of $u_{t}$, $\ddbar u_{t}\leq 0$. Using the monotonicity of $F$ we have $\psi_{t}\leq \tilde{\psi}$. That is, $e^{b_{t}}\leq \psi^{-t}\tilde{\psi}^{t}\leq 1$ since $\tilde{\psi}\leq \psi$. This proves the claim.

Similarly, at the minimum point of $u_{t}$, $\ddbar u_{t}\geq 0$. We can also obtain a lower bound for $b_{t}$, this yields a uniform bound for $b_{t}$.

If we set $h_{t}$ by $\psi_{t}={n\choose m}h_{t}^{m}$, by the previous claim we have $h_{t}\leq h.$ Then the hypersurface $\Gamma^{h}$ lies above $\Gamma^{h_{t}}$.  Therefore, $\underline{u}$ is still  the $\mathcal{C}$-subsolution for the flow \eqref{flow}. Hence, we have uniform $C^{\infty}$ estimates for $u_{t}$. 

Define a nonempty set (including $0$)
\begin{equation}\label{}
\mathcal{T}=\big\{t'\in[0,1]:~\exists u_{t}\in C^{3,\alpha}(M), b_{t} ~\textrm{solves}~ \eqref{flow} ~\textrm{for}~ t\in[0,t']\big\}.
\end{equation}
To solve the equation \eqref{dp}, it suffices to check $\mathcal{T}$ is both closed and open.
The closedness is easily from the uniform bounds for $b_{t}$ and $u_{t}\in C^{3,\alpha}(M)$.

Now we shall prove $\mathcal{T}$ is open. That is, for each $\hat{t}\in \mathcal{T}$, we have $[\hat{t},\hat{t}+\delta)\subset \mathcal{T}$ for some $\delta>0$.
Set $G(u)=\frac{\chi_{v+u}^{n}}{\chi_{v+u}^{n-m}\wedge \omega^{m}}$ on $\mathcal{H}(M,\chi_{v})$. Define an almost Hermitian metric
\[\Omega=\sqrt{-1}\sum_{i,j}G_{i\bar{j}}(u_{\hat{t}})\theta_{i}\wedge\bar{\theta}_{j}.\]
Here $G^{i\bar{j}}=\frac{\partial \log G}{\partial u_{i\bar{j}}}$ and $\{G_{i\bar{j}}\}$ is the inverse matrix of $\{G^{i\bar{j}}\}$. By Cauduchon's work (see \cite[Theorem 2.1]{CTW} for instance), there exists a potential function $\phi\in C^{\infty}(M)$ such that $e^{\phi}\Omega$ is a Gauduchon metric. We may normalize $\phi$ by adding a constant such that $\int_{M}e^{(n-1)\phi}\Omega^{n}=1$.

Note that the flow \eqref{flow} on $[\hat{t},\hat{t}+\delta)$ is equivalent to
\begin{equation}\label{IFT}
G(u_{t})=G(u_{\hat{t}})\psi^{t-\hat{t}}\tilde{\psi}^{-t+\hat{t}}e^{c_{t}}
\Big(\int_{M}\frac{G(u_{t})}{G(u_{\hat{t}})}e^{(n-1)\phi}\Omega^{n}\Big)
\end{equation}
for some constant $c_{t}$ satisfies
\[\int_{M}\psi^{t-\hat{t}}\tilde{\psi}^{-t+\hat{t}}e^{c_{t}}e^{(n-1)\phi}\Omega^{n}=1.\]
Define a map $\Psi$ by
\[\Psi(\eta)=\log \frac{G(u_{\hat{t}}+\eta)}{G(u_{\hat{t}})}-\log\Big(\int_{M}\frac{G(u_{\hat{t}}+\eta)}{G(u_{\hat{t}})}e^{(n-1)\phi}\Omega^{n}\Big),\]
which maps $\eta\in C^{3,\alpha}(M)$ with $\int_{M}\eta e^{(n-1)\phi}\Omega^{n}=0$ and $\chi_{\eta+u_{\hat{t}}+v}>0$ to $\Psi(\eta)\in C^{1,\alpha}$ with $\int_{M}e^{\Psi(\eta)+(n-1)\phi}\Omega^{n}=1$. Observe that $\Psi(0)=0$, and the linearization of $\Psi$ at $\eta=0$ is given by
\[D\Psi(0)(\xi)=\frac{n\Omega^{n-1}\wedge \ddbar \xi}{\Omega^{n}}=\Delta_{\Omega}\xi,\]
since $\int_{M}\ddbar \xi\wedge e^{(n-1)\phi}\Omega^{n-1}=0$. By another result of Gauduchon (see \cite[Theorem 2.2]{CTW}), the Laplacian operator $\Delta_{\Omega}$ is invertible. Now by inverse function theorem, we can solve \eqref{IFT} on $[\hat{t},\hat{t}+\delta)$ for some $\delta>0$ small.
This completes the proof of the theorem.

\end{CJK}
\end{document}